\documentclass[10pt]{article}

\usepackage{amsfonts}
\usepackage{amsmath}
\usepackage{amssymb}
\usepackage{amsthm}
\usepackage{enumerate}
\usepackage{mdwlist}
\usepackage{mathrsfs}
\usepackage{bussproofs}
\usepackage{graphicx}
\usepackage{latexsym}
\usepackage{titlesec}
\usepackage{stmaryrd}
\usepackage{verbatim}

\titleformat*{\section}{\large\bfseries}
\titleformat*{\subsection}{\normalsize\bfseries}
\titleformat*{\subsubsection}{\normalsize\em}

\newtheorem{theorem}{Theorem}[section]
\newtheorem{proposition}[theorem]{Proposition}
\newtheorem{lemma}[theorem]{Lemma}
\newtheorem{corollary}[theorem]{Corollary}

\theoremstyle{remark}
\newtheorem{definition}[theorem]{Definition}
\newtheorem{remark}[theorem]{Remark}
\newtheorem{example}[theorem]{Example}

\usepackage{amsfonts}
\usepackage{amsmath}
\usepackage{enumerate}
\usepackage{url}

\def\NN{\mathbb{N}}

\def\QQ{\mathbb{Q}}
\def\RR{\mathbb{R}}

\newcommand{\norm}[1]{\|{#1}\|}

\newcommand{\seq}[1]{\{#1\}}
\newcommand{\ip}[1]{\langle {#1} \rangle}
\newcommand{\sgn}{\mathrm{sgn}}

\begin{document}

\title{A computational study of a class of recursive inequalities}

\author{Morenikeji Neri and Thomas Powell}

\maketitle

\begin{abstract}    % type your abstract below
We examine the convergence properties of sequences of nonnegative real numbers that satisfy a particular class of recursive inequalities, from the perspective of proof theory and computability theory. We first establish a number of results concerning rates of convergence, setting out conditions under which computable rates are possible, and when not, providing corresponding rates of metastability. We then demonstrate how the aforementioned quantitative results can be applied to extract computational information from a range of proofs in nonlinear analysis. Here we provide both a new case study on subgradient algorithms, and give overviews of a selection of recent results which each involve an instance of our main recursive inequality. This paper contains the definitions of all relevant concepts from both proof theory and mathematical analysis, and as such, we hope that it is accessible to a general audience.
\end{abstract}

%%%%%%%%%%%%%%%%%%%%%%%%%%%%%%%%%%%%%%%%%%%%%%%%%%%%%%%%%%%%%%%%%%%%%%%%%%%%%%%%%%%%%%%%%%
%%%%%%%%%%%%%%%%%%%%%%%%%%%%%%%%%%%%%%%%%%%%%%%%%%%%%%%%%%%%%%%%%%%%%%%%%%%%%%%%%%%%%%%%%%
\section{Introduction}
\label{sec:intro}
%%%%%%%%%%%%%%%%%%%%%%%%%%%%%%%%%%%%%%%%%%%%%%%%%%%%%%%%%%%%%%%%%%%%%%%%%%%%%%%%%%%%%%%%%%
%%%%%%%%%%%%%%%%%%%%%%%%%%%%%%%%%%%%%%%%%%%%%%%%%%%%%%%%%%%%%%%%%%%%%%%%%%%%%%%%%%%%%%%%%%

Recursive inequalities on sequences of nonnegative real numbers play an important role in functional analysis. They can be used to establish convergence properties of algorithms in a very general setting, and often yield explicit \emph{rates} of convergence in addition. A simple example of this phenomenon is represented by the inequality 
\begin{equation}
\label{eqn:recineq:banach}
	\mu_{n+1}\leq c\mu_n
\end{equation}
for $c\in [0,1)$. Here, any sequence $\seq{\mu_n}$ of nonnegative reals that satisfies (\ref{eqn:recineq:banach}) converges to zero, and moreover an effective rate of convergence for $\mu_n\to 0$ is given by $\mu_n \leq c^n\mu_0$. The inequality (\ref{eqn:recineq:banach}) is associated most famously with the Banach fixed point theorem: Suppose that $(X,d)$ is a metric space and $T:X\to X$ a contractive mapping with constant $c$ i.e.
\begin{equation*}
d(T(x),T(y))\leq c d(x,y)
\end{equation*}
for all $x,y\in X$. If $x^\ast$ is a fixpoint of $T$, and $\seq{x_n}$ the algorithm defined by $x_{n+1}:=Tx_n$ for some starting value $x_0\in X$, it is easy to see that $\mu_n:=d(x_n,x^\ast)$ satisfies (\ref{eqn:recineq:banach}), since
\begin{equation*}
d(x_{n+1},x^\ast)=d(T(x_n),T(x^\ast))\leq cd(x_n,x^\ast)
\end{equation*}
Therefore $x_n\to x^\ast$ follows from $\mu_n\to 0$, and moreover, the same rate of convergence applies, with $d(x_n,x^\ast)\leq c^nd(x_0,x^\ast)$. 

What we have just sketched is a simple but representative example of a much broader strategy for proving that an algorithm $\seq{x_n}$ in some space $X$ converges to a point $x^\ast$ (where the latter could be, for example, a minimiser of a function, a fixpoint of an operator, or a solution of an equation), namely:
\begin{enumerate}

	\item Show that $\mu_n:=d(x_n,x^\ast)$ satisfies some recursive inequality,
	
	\item Prove that any sequence $\seq{\mu_n}$ of nonnegative reals satisfying that recursive inequality converges to zero.

\end{enumerate}
An important advantage of this strategy is that individual classes of recursive inequalities can typically be used to prove the convergence of a wide range of different algorithms. Moreover, as a happy byproduct, additional properties pertaining to the recursive inequality that we are able to uncover, such as a rate of convergence, often then apply uniformly to any algorithm that can be reduced to that inequality.

The purpose of this article is to study several classes of sequences that all satisfy the following recursive inequality:
\begin{equation}
\label{eqn:basic:intro}\tag{$\star$}
\mu_{n+1}\leq \mu_n-\alpha_n\beta_n+\gamma_n
\end{equation}
where we assume throughout that $\seq{\mu_n}$, $\seq{\alpha_n}$, $\seq{\beta_n}$ and $\seq{\gamma_n}$ are sequences of nonnegative real numbers, with the additional conditions that:
\begin{itemize}

	\item $\sum_{i=0}^\infty\alpha_i=\infty$, and
	
	\item $\gamma_n\to 0$ as $n\to\infty$.

\end{itemize}
The significance of (\ref{eqn:basic:intro}) as a basic characterising property for a range of interrelated convergent sequences has been exploited in several places in the mathematics literature -- see for example Section 2 of Alber and Iusem \cite{alber-iusem:01:subgradient}. We give a detailed description of the mathematical relevance of (\ref{eqn:basic:intro}) in Section \ref{sec:prelim:analysis} below, but informally speaking, the $\seq{\alpha_n}$ are typically step sizes for some algorithm, while $\seq{\gamma_n}$ represents an ``error-term'' that eventually vanishes. We are then interested in the convergence properties of $\seq{\mu_n}$ and $\seq{\beta_n}$. As we will see, the aforementioned conditions on $\seq{\alpha_n}$ and $\seq{\gamma_n}$ alone are not sufficient to establish convergence of either $\seq{\mu_n}$ or $\seq{\beta_n}$, but their behaviour can be controlled by imposing a range of additional conditions, which then divide (\ref{eqn:basic:intro}) into a number of interesting subclasses that each have distinct properties and applications. While (\ref{eqn:recineq:banach}) represents a very simple case of (\ref{eqn:basic:intro}) for $\beta_n:=\mu_n$, $\alpha_n:=1-c$ and $\gamma_n:=0$, in general, proving convergence of $\seq{\mu_n}$ and $\seq{\beta_n}$ can be extremely subtle.

%%%%%%%%%%%%%%%%%%%%%%%%%%%%%%%%%%%%%%%%%%%%%%%%%%%%%%%%%%%%%%%%%%%%%%%%%%%%%%%%%%%%%%%%%%
\subsection*{Aims and overview of the paper}
\label{sec:intro:overview}
%%%%%%%%%%%%%%%%%%%%%%%%%%%%%%%%%%%%%%%%%%%%%%%%%%%%%%%%%%%%%%%%%%%%%%%%%%%%%%%%%%%%%%%%%%

For us, the relevance of (\ref{eqn:basic:intro}) lies primarily in the fact that it has implicitly appeared in several recent papers in applied proof theory, where classes of sequences satisfying (\ref{eqn:basic:intro}) have been analysed from a quantitative perspective in order to extract rates of convergence for concrete problems in nonlinear analysis. 

Applied proof theory is an area of research that uses ideas and techniques from proof theory to produce new results in mainstream mathematics. While the field has historical roots in Hilbert’s program and Kreisel's idea of `unwinding' proofs \cite{kohlenbach:pp:kreisel,kreisel:51:proofinterpretation:part1,kreisel:52:proofinterpretation:part2}, its emergence as a powerful area of applied logic only started in the early 2000s with the research of Kohlenbach and his collaborators. This work resulted in numerous applications across different fields including nonlinear analysis, approximation theory, ergodic theory and convex optimization, where proof theoretic methods have been used to obtain both new quantitative results along with qualitative generalisations of existing theorems. In parallel, new logical systems for reasoning about abstract mathematical structures have been developed, and used to formulate general logical metatheorems that explain the aforementioned applications as instances of logical phenomena (beginning with Kohlenbach \cite{kohlenbach:05:metatheorems} and most recently by Pischke \cite{pischke:pp:metatheorem}). A comprehensive overview of the ideas and techniques that underlie applied proof theory, along with a survey of case studies up until 2008, can be found in the textbook by Kohlenbach \cite{kohlenbach:08:book}, and an account of some applications in the subsequent decade is provided by the review papers from the same author \cite{kohlenbach:17:recent,kohlenbach:19:nonlinear:icm}. Just some examples of recent applications in new areas include pursuit-evasion games (Kohlenbach, L\'opez-Acedo and Nicolae \cite{kohlenbach-lopezacedo-nicolae:21:lionman}), differential algebra (Simmons and Towsner \cite{simmons:towsner:19:polyrings}), and probability theory (Arthan and Oliva \cite{arthan-oliva:21:borel-cantelli}).

Many papers in the applied proof theory literature deal with recursive inequalities of some kind, from authors including Cheval, Dinis, Kohlenbach, K\"ornlein, Lambov, Leu\c{s}tean, L\'opez-Acedo, Nicolae, Pinto, Sipo\c{s}, Wiesnet, and the second author, and it is these are most relevant to the work presented below. A more detailed summary of recursive inequalities in the literature, along with the relevant citations, is given in Section \ref{sec:applications:lit}, where we connect these results with the relevant subclasses of (\ref{eqn:basic:intro}) that are explored in this paper.

The main contributions of our paper are to generalise some abstract convergence results on recursive inequalities and classify them within a overarching framework, establish new abstract quantitative lemmas for sequences of nonnegative reals, and present new applications.

Our paper is divided into two main parts. In Section \ref{sec:recineq} we focus exclusively on sequences of real numbers satisfying (\ref{eqn:basic:intro}), and address the following questions:
\begin{enumerate}[(a)]

	\item\label{q1} Under which conditions on $\seq{\alpha_n}$ and $\seq{\gamma_n}$ can we prove that $\seq{\mu_n}$ and $\seq{\beta_n}$ converge, and when do we obtain computable rates of convergence?
	
	\item\label{q2} Where computable rates of convergence are not possible, can we obtain computable bounds on the corresponding rates of \emph{metastable} convergence for these sequences?

\end{enumerate}
Answers to these questions are both interesting and useful. Where (\ref{q1}) can be answered positively, we are then in a position to extract concrete rates of convergence from more complex proofs which reduce to that particular instance of (\ref{eqn:basic:intro}) as a lemma. When the answer is negative, we have demonstrated that there is no \emph{uniform} method for obtaining computable rates in this way, and this can also be revealing: Quite often, analysts who use recursive inequalities of the latter type simply observe that their proof does not yield convergence rates. Our results then provide an explanation as to why that is the case, using tools from computable analysis such as Specker sequences. On the other hand, for all cases that we consider we are able to answer (\ref{q2}) positively, and provide a corresponding quantitative lemma that yields rates of metastability, given appropriate metastable rates for assumptions. This allows us to extract quantitative information from proofs where computable rates of convergence do not seem to be possible.

Section \ref{sec:applications} then focuses on applications, giving concrete examples of the phenomena we have just described. We begin with a brief overview of recursive inequalities in the applied proof theory literature, and then present a new case study which uses some of our new results on recursive inequalities in the context of subgradient methods, to firstly explain why rates of convergence cannot be extracted from the kind of convergence proofs given in the literature, and secondly providing instead a rate of metastability for a class of subgradient algorithms. This is followed by a high-level survey of some selected case studies, where we demonstrate how our abstract numerical results can be more broadly utilised in a wide range of different settings.

Our main technical contributions include the following:
\begin{itemize}

	\item (Real analysis) Strengthenings of existing convergence results on sequences of real numbers, with necessary and sufficient conditions for convergence for both main classes studied (Proposition \ref{prop:recineq1:beta:generalise} and Proposition \ref{prop:recineq2});
	
	\item (Computability theory) Negative results showing that for sequences of reals satisfying (\ref{eqn:basic:intro}), computable rates of convergence are in general not possible without additional assumptions (Theorem \ref{thm:block} and Proposition \ref{prop:recineq2:specker});
	
	\item (Proof theory) The extraction of computable rates of metastability for all convergence results considered (Theorems \ref{thm:recineq1:beta:comp}, \ref{thm:recineq2:metastable} and \ref{thm:recineq2:metastable:v2}), along with rates of convergence in special cases (Corollaries \ref{cor:recineq1:beta:rates}, \ref{cor:recineq2} and \ref{cor:recineq2b}).
	
	\item (Applied proof theory) An abstract characterisation of a class of gradient descent methods along with a general quantitative convergence result for these algorithms (Theorem \ref{thm:abstract:gradient}), and application to a projective subgradient method of Alber, Iusem and Solodov \cite{alber-iusem-solodov:98:subgradient} (Theorem \ref{thm:projective}).

\end{itemize}
In addition to this technical work, our secondary objective has been to write a paper that serves as a useful and accessible introduction to some recent research in applied proof theory. Quantitative versions of lemmas involving recursive inequalities have become increasingly common over the last 5--10 years, and while this article is by no means intended as a comprehensive overview of all of these, by focusing on a relevant class of inequalities we hope to provide some general insight into the role they play in proof theoretic approaches to analysis.  Our work coincides nicely with a recent survey paper on convergent sequences by Franci and Grammatico \cite{franci-grammatico:convergence:survey:22}, which provides an extensive overview of their many applications over different areas of analysis and probability theory. Our paper is similar in spirit (though here with a focus on proof theory and computability theory) in presenting an overview of several variations a type of recursive inequality, and illustrating how these are applied in concrete areas.

With all this in mind, our aim has been to write for as general an audience as possible. All of the core definitions and concepts from both proof theory and computability theory are included in Section \ref{sec:prelim:logic}, and our main abstract results in Section \ref{sec:recineq} do not require any knowledge of analysis beyond basic facts about sequences and series. The more advanced areas of mathematical analysis covered in Section \ref{sec:applications} are carefully motivated, with simple illustrative examples provided before each of the main results, and it should be possible for readers who are not experts in the relevant subfields of analysis to follow these with no prior knowledge beyond some elementary facts about inner products and normed spaces.

Some of this research has been formalized in the Lean theorem prover\footnote{\url{https://leanprover.github.io/}}, as part of a broader strategy of building a library for key definitions and lemmas from applied proof theory. We conclude the paper with a brief account of this effort, arguing why recursive inequalities form such a good starting point for an applied proof theory library, and by outlining some potentially interesting directions for future work.

%%%%%%%%%%%%%%%%%%%%%%%%%%%%%%%%%%%%%%%%%%%%%%%%%%%%%%%%%%%%%%%%%%%%%%%%%%%%%%%%%%%%%%%%%%
%%%%%%%%%%%%%%%%%%%%%%%%%%%%%%%%%%%%%%%%%%%%%%%%%%%%%%%%%%%%%%%%%%%%%%%%%%%%%%%%%%%%%%%%%%
\section{Preliminaries}
\label{sec:prelim}
%%%%%%%%%%%%%%%%%%%%%%%%%%%%%%%%%%%%%%%%%%%%%%%%%%%%%%%%%%%%%%%%%%%%%%%%%%%%%%%%%%%%%%%%%%
%%%%%%%%%%%%%%%%%%%%%%%%%%%%%%%%%%%%%%%%%%%%%%%%%%%%%%%%%%%%%%%%%%%%%%%%%%%%%%%%%%%%%%%%%%

This paper is about both mathematical analysis and logic, and we divide this background section into two distinct parts, focusing on each in turn. We begin by motivating the recursive inequality (\ref{eqn:basic:intro}) in a little more detail, justifying why we have chosen it as an object of study and preparing for the applications that will be presented in Section \ref{sec:applications}. We then collect together some key facts and definitions from proof theory and computability theory that will be used in what follows.

%%%%%%%%%%%%%%%%%%%%%%%%%%%%%%%%%%%%%%%%%%%%%%%%%%%%%%%%%%%%%%%%%%%%%%%%%%%%%%%%%%%%%%%%%%
\subsection{Mathematical context}
\label{sec:prelim:analysis}
%%%%%%%%%%%%%%%%%%%%%%%%%%%%%%%%%%%%%%%%%%%%%%%%%%%%%%%%%%%%%%%%%%%%%%%%%%%%%%%%%%%%%%%%%%

At the very beginning of the paper we gave, via Banach's fixed point theorem, a simple example of how recursive inequalities underlie convergence proofs. We now explain how the more complex recursive inequality (\ref{eqn:basic:intro}) appears in nonlinear analysis, first through a simple example and then with a high-level overview, detailing the role that each of the individual components usually play. 

\subsubsection{A simple example: weakly contractive mappings and the sine function}

For purely illustrative purposes, let us consider a simple class of mappings known as $\psi$-weakly contractive mappings, introduced by Alber and Guerre-Delabriere in \cite{alber-guerredelabriere:97:weaklycontractive} and widely studied since. Letting $X$ be some normed space and $\psi:[0,\infty)\to [0,\infty)$ a nondecreasing function with $\psi(0)=0$ and $\psi(t)>0$ for $t>0$, a mapping $T:X\to X$ is called $\psi$-weakly contractive on some $C\subseteq X$ if it satisfies
\begin{equation}
\label{eqn:wc}
\norm{Tx-Ty}\leq \norm{x-y}-\psi(\norm{x-y})
\end{equation}
for all $x,y\in C$. Any contractive mapping on $C$ is also $\psi$-weakly contractive for $\psi(t)=(1-c)t$ with $c\in [0,1)$, but the $\psi$-weakly contractive mappings are much more general. For example, setting $X=\RR$, we can show, by considering its Taylor expansion, that 
\begin{equation*}
|\sin(x)-\sin(y)|\leq |x-y|-\frac{|x-y|^3}{8}
\end{equation*}
for $x,y\in [0,1]$, and so the sine function is $\psi$-weakly contractive on $[0,1]$ for $\psi(t)=t^3/8$ (for details see \cite[p.14]{alber-guerredelabriere:97:weaklycontractive}). 

Now suppose that $x^\ast$ is a fixed point of a $\psi$-weakly contractive mapping $T$. A classic algorithm for approximating fixed points is the so-called Krasnoselskii-Mann (KM) scheme, defined recursively from some starting point $x_0$ by
\begin{equation}
\label{eqn:km}
x_{n+1}=(1-\alpha_n)x_n+\alpha_n Tx_n
\end{equation}
where $\seq{\alpha_n}$ is a sequence of nonnegative real numbers with $\sum_{i=0}^\infty\alpha_i=\infty$. Again, this is a generalisation of the scenario considered in the introduction, where the algorithm $x_{n+1}=Tx_n$ corresponds to the special case $\alpha_n=1$. Using (\ref{eqn:wc}) and (\ref{eqn:km}) along with the triangle inequality, we can calculate that
\begin{equation*}
\begin{aligned}
\norm{x_{n+1}-x^\ast}&=\norm{(1-\alpha_n)x_n+\alpha_n Tx_n-x^\ast}\\
&=\norm{(1-\alpha_n)(x_n-x^\ast)+\alpha_n(Tx_n-x^\ast)}\\
&\leq (1-\alpha_n)\norm{x_n-x^\ast}+\alpha_n\norm{Tx_n-Tx^\ast}\\
&\leq (1-\alpha_n)\norm{x_n-x^\ast}+\alpha_n(\norm{x_n-x^\ast}-\psi(\norm{x_n-x^\ast}))\\
&=\norm{x_n-x^\ast}-\alpha_n\psi(\norm{x_n-x^\ast})
\end{aligned}
\end{equation*}
and thus $\mu_n:=\norm{x_n-x^\ast}$ satisfies the following special case of (\ref{eqn:basic:intro}) for $\beta_n:=\psi(\mu_n)$ and $\gamma_n:=0$:
\begin{equation}
\label{eqn:recineq:simple}
\mu_{n+1}\leq \mu_n-\alpha_n\psi(\mu_n)
\end{equation}
As we will see in Section \ref{sec:recineq}, it can be shown that if $\sum_{i=0}^\infty \alpha_i=\infty$, then whenever $\seq{\mu_n}$ satisfies (\ref{eqn:recineq:simple}) we have $\mu_n\to 0$. There also exist effective rates of convergence for $\mu_n\to 0$ in terms of the speed of divergence of $\sum_{i=0}^\infty \alpha_i$. These will be given in a general setting in Section \ref{sec:recineq} and are also formulated in a slightly different way in Alber and Guerre-Delabriere\cite{alber-guerredelabriere:97:weaklycontractive}.

As a consequence, by focusing on the convergence behaviour of sequences of reals satisfying (\ref{eqn:recineq:simple}), we can show that whenever $x^\ast$ is a fixed point of some $\psi$-weakly contractive mapping $T$, the corresponding KM algorithm converges to this fixed point with a known rate of convergence. In particular, setting $T=\sin$ and $\alpha_n=1$, we can show that $x_{n+1}=\sin x_n$ converges to zero for any $x_0\in [0,1]$ and provide an explicit convergence speed.

In Section \ref{sec:applications:weakly:contractive} we will give a slightly more intricate asymptotic variant of being $\psi$-weakly contractive, which leads to a version of (\ref{eqn:recineq:simple}) with a nonzero error terms $\seq{\gamma_n}$.

\subsubsection{The general setting}

Moving away from our concrete example to the general setting, the recursive inequality (\ref{eqn:basic:intro}) is applied in a broadly similar way across each of the different settings we explore in this paper. We typically have a mapping $T$ on some normed space $X$, a special value $x^\ast$ (such as a minimizer or a fixed point), and define an algorithm $\seq{x_n}$ in terms of some step size $\seq{\alpha_n}$ with $\sum_{i=0}^\infty \alpha_i=\infty$, and we show that
\begin{equation*}
\norm{x_{n+1}-x^\ast}\leq \norm{x_n-x^\ast}-\alpha_n F(x_n)+\gamma_n
\end{equation*}
where $F:X\to \RR$ is some function, $\seq{\gamma_n}$ a sequence of error terms with $\gamma_n\to 0$, and both $F$ and $\seq{\gamma_n}$ are defined in terms of the mapping $T$ or the space $X$. We are then usually interested in the convergence behaviour of $\norm{x_n-x^\ast}$ or $F(x_n)$ (i.e. $\mu_n$ or $\beta_n$), and this leads naturally to the two questions (\ref{q1}) and (\ref{q2}) posed in the introduction.

The first part of this paper focuses exclusively on convergence properties of sequences of real numbers, but in Section \ref{sec:applications} we outline several concrete applications of these convergence results, which each adhere to the above phenomenon, namely:
\begin{itemize}

	\item $\seq{x_n}$ is a generalised gradient descent algorithm, $T:X\to \RR$ is a convex, continuous function and $x^\ast$ a minimiser of $T$ (a new application discussed in Section \ref{sec:applications:gradient:projective}).
	
	\item $\seq{x_n}$ a modified Krasnoselkii-Mann scheme, $T:X\to X$ is an asymptotically $\psi$-weakly contractive mapping and $x^\ast$ is a fixed point of $T$ (surveyed in Section \ref{sec:applications:weakly:contractive}).
	
	\item $\seq{x_n}$ is an implicit scheme, $T:X \to 2^X$ a set-valued accretive operator and $x^\ast$ a zero of $T$ (surveyed in Section \ref{sec:applications:acc}).

\end{itemize}
These in turn are just representative examples of a much wider variety of settings in which recursive inequalities of the form (\ref{eqn:basic:intro}) have been applied.

%%%%%%%%%%%%%%%%%%%%%%%%%%%%%%%%%%%%%%%%%%%%%%%%%%%%%%%%%%%%%%%%%%%%%%%%%%%%%%%%%%%%%%%%%%
\subsection{Preliminaries from proof theory and computable analysis}
\label{sec:prelim:logic}
%%%%%%%%%%%%%%%%%%%%%%%%%%%%%%%%%%%%%%%%%%%%%%%%%%%%%%%%%%%%%%%%%%%%%%%%%%%%%%%%%%%%%%%%%%

We now turn our attention to the relevant background from logic, outlining some key definitions and results which will be important in what follows. We let $\QQ_+$ denote the set of strictly positive rational numbers.
\begin{definition}
\label{def:roc}
Let $\seq{a_n}$ be a sequence of real number that converges to some limit $a$. A rate of convergence for $a_n\to a$ is any function $\phi:\QQ_+\to \NN$ satisfying
\begin{equation*}
\forall \varepsilon\in \QQ_+\, \forall n\geq \phi(\varepsilon)\, (|a_n-a|\leq \varepsilon)
\end{equation*}
Similarly, $\phi$ is a rate of Cauchy convergence for $\seq{a_n}$ if
\begin{equation*}
\forall \varepsilon\in \QQ_+\, \forall m,n\geq \phi(\varepsilon)\, (|a_n-a_m|\leq \varepsilon)
\end{equation*}
We say that $\seq{a_n}$ has a computable rate of (Cauchy) convergence if it possesses a rate of (Cauchy) convergence $\phi$ that is computable in the usual sense. Any computable rate of convergence can be converted to a computable rate of Cauchy convergence, and vice-versa.
\end{definition}
While the existence of a (not necessarily computable) rate of convergence is logically equivalent to convergence itself, it is well known that there exist convergent sequences of computable numbers that do not have a computable rate of convergence. The canonical examples of this phenomenon are the so-called Specker sequences:
\begin{proposition}
[Specker 1949 \cite{specker:49:sequence}]\label{prop:specker} There exists a computable, monotonically increasing, bounded sequence $\seq{a_n}$ of rational numbers whose limit $a$ is not a computable real number, and thus no computable rate of convergence for $a_n\to a$ exists. 
\end{proposition} 
It is not necessarily the case that having a computable limit guarantees the existence of a rate of convergence either: Provided we no longer insist on monotonicity, we can construct a computable sequence of rationals that converges to $0$ but with no computable rate. We now state and prove a slightly generalized version of what is essentially a folklore construction, which will be important in Section \ref{sec:recineq}.
\begin{proposition}
\label{prop:specker:zero}
Let $\seq{a_n}$ be any strictly decreasing computable sequence of positive rationals that converges to $0$. Letting $T_m$ denote the Turing machine with index $m$, define
\begin{equation*}
s_n:=\begin{cases}
	a_m & \parbox[t]{.6\textwidth}{\rm for the minimum $m\leq n$ such that $T_m$ halts on input $m$ in exactly $n$ steps}\\[0.5cm]
	0 & \mbox{\rm if no such $m$ exists}
\end{cases}
\end{equation*}
Then $s_n\to 0$ but has no computable rate of convergence.
\end{proposition}

\begin{proof}
Fix $n\in\NN$ and let $N$ be such that any of the machines $T_i$ which terminate on input $i$ for $i=0,\ldots,n-1$ do so in at most $N$ steps. Then for any $k\geq N+1$ we have $s_k\leq a_n$: Were this not the case then $s_k=a_m$ where $m\leq k$ is the least such that $T_m$ halts on input $m$ in exactly $k$ steps. But since $\seq{a_n}$ is strictly decreasing, we must have $m<n$ and therefore $k\leq N$. Since $a_n\to 0$ this therefore implies that $s_n\to 0$.

Now suppose for contradiction that $s_n\to 0$ with some computable rate of convergence $\phi$. We argue that for any $k\in\NN$, if $T_k$ halts on input $k$, then it does so in less than $\psi(k):=\max\{k,\phi(a_{k+1})\}$ steps: If not and $T_k(k)$ halts in $n$ steps for $\psi(k)\leq n$, then since $k\leq n$ we must have $s_n=a_m$ for some $m\leq k$ (i.e. the minimal $m\leq n$ such that $T_m(m)$ halts in exactly $n$ steps). Thus $s_n=a_m\geq a_k>a_{k+1}$. But since $\phi$ is a rate of convergence for $s_n\to 0$ and $\phi(a_{k+1})\leq n$ then $s_n\leq a_{k+1}$, a contradiction. Therefore if $\phi$ were computable then $\psi(k)$ forms a computable upper bound on the number of steps it takes $T_k$ to halt on input $k$, contradicting the unsolvability of the halting problem. 
\end{proof}
In scenarios where we are, in general, unable to provide a computable rate of convergence, we can often nevertheless produce computable bounds for the following classically equivalent reformulation, known as \emph{metastable} convergence.
\begin{equation}
\label{eqn:metastable}
\forall \varepsilon\in\QQ_+\, \forall g:\NN\to \NN\,  \exists n\, \forall k\in [n,n+g(n)](|a_k-a|\leq \varepsilon) 
\end{equation}
where $[a,a+b]:=\{a,a+1,\ldots,a+b\}$. Note that (\ref{eqn:metastable}) easily follows from $a_n\to a$ since if $n$ is such that $|a_k-a|\leq \varepsilon$ for all $k\geq n$, then we also have $|a_k-a|\leq \varepsilon$ for $k\in [n,n+g(n)]$. The other direction, on the other hand, requires us to reason by contradiction: If $a_n\nrightarrow a$ then there exists some $\varepsilon>0$ such that
\begin{equation*}
\forall n\, \exists\, k\geq n\, (|a_k-a|>\varepsilon)
\end{equation*}
and therefore there exists some function $g:\NN\to\NN$ bounding $k$ in $n$ i.e.
\begin{equation*}
\forall n\, \exists k\in [n,n+g(n)](|a_k-a|>\varepsilon)
\end{equation*}
which contradicts (\ref{eqn:metastable}). The following notion of a bound on $n$ in (\ref{eqn:metastable}) in terms of $\varepsilon$ and $g$ appears throughout the applied proof theory literature:
\begin{definition}
A rate of metastability is a functional $\Phi:\QQ_+\times (\NN\to\NN)\to\NN$ bounding $n$ in (\ref{eqn:metastable}) i.e.
\begin{equation*}
\forall \varepsilon\in\QQ_+\, \forall g:\NN\to \NN\,  \exists n\leq \Phi(\varepsilon,g)\, \forall k\in [n,n+g(n)](|a_k-a|\leq \varepsilon) 
\end{equation*}
There is an analogous notion of metastability for Cauchy convergence, namely
\begin{equation*}
\forall \varepsilon\in\QQ_+\, \forall g:\NN\to \NN\,  \exists n\, \forall i,j\in [n,n+g(n)](|a_i-a_j|\leq \varepsilon) 
\end{equation*}
and rates of metastability are defined similarly:
\begin{equation*}
\forall \varepsilon\in\QQ_+\, \forall g:\NN\to \NN\,  \exists n\leq \Phi(\varepsilon,g)\, \forall i,j\in [n,n+g(n)](|a_i-a_j|\leq \varepsilon) 
\end{equation*}
\end{definition}
A detailed account of metastability is given by Kohlenbach in \cite{kohlenbach:08:book} (cf. Chapter 2 in particular). From the perspective of logic, metastability (\ref{eqn:metastable}) is just the Herbrand normal form of convergence, and a rate of metastability a solution to the so-called `no-counterexample interpretation' of the convergence statement as set out by Kreisel \cite{kreisel:51:proofinterpretation:part1, kreisel:52:proofinterpretation:part2}. However, metastable convergence theorems have been independently studied in mathematics: In particular, metastable versions of the monotone convergence principle and other `infinitary' statements have been discussed by Tao in \cite{tao:07:softanalysis,tao:08:ergodic}, and have found nontrivial applications in several areas of mathematics. Highly uniform and computable rates of metastability can often be obtained convergence proofs, even when those proofs are nonconstructive and computable rates of convergence are not possible. Indeed, the extraction of rates of metastability using proof theoretic techniques is a standard result in applied proof theory (examples from just the present year include Cheval, Kohlenbach and Leu\c{s}tean \cite{cheval-kohlenbach-leustean:23:halpern}, Freund and Kohlenbach \cite{freund-kohlenbach:23:ergodic}, and the second author \cite{powell:23:littlewood}).

A number of the results that we study involve convergence or divergence properties of infinite series. We note that in the case that $\sum_{i=0}^\infty b_i<\infty$ where $\seq{b_i}$ a sequence of nonnegative reals, a rate $\phi$ of Cauchy convergence for $\sum_{i=0}^\infty b_i$ satisfies
\begin{equation*}
\forall \varepsilon\in \QQ_+\, \forall n,m\in\NN\, \left( n\geq \phi(\varepsilon)\implies \sum_{i=n}^{n+m}b_i\leq\varepsilon\right)
\end{equation*}
while a corresponding rate of metastability $\Phi$ satisfies
\begin{equation*}
\forall \varepsilon\in \QQ_+\,\forall g:\NN\to\NN\, \exists n\leq \Phi(\varepsilon,g)\left(\sum_{i=n}^{n+g(n)}b_i\leq\varepsilon\right)
\end{equation*}
Finally, for infinite series that diverge, we will need to introduce a computational analogue of this property. The following notion was first used by Kohlenbach in \cite{kohlenbach:01:borweinreichshafrir}:
\begin{definition}
\label{def:rod}
Suppose that $\seq{b_n}$ is a sequence of nonnegative real numbers such that $\sum_{i=0}^\infty b_i=\infty$. A rate of divergence for $\sum_{i=0}^\infty b_i=\infty$ is any function $r:\NN\times \QQ_+\to\NN$ such that
\begin{equation*}
\forall n\in\NN\, \forall x\in \QQ_+\left(\sum_{i=n}^{r(n,x)}b_i\geq x\right)
\end{equation*}
For simplicity in formulating our bounds, we also assume that rates of divergence are monotone in their first argument i.e.
\begin{equation*}
m\leq n\implies r(m,x)\leq r(n,x)
\end{equation*}
for all $m,n\in\NN$ and $x\in \QQ_+$.
\end{definition}

\begin{remark}
\label{rem:rod}
Our monotonicity assumption is harmless: Any (potentially non-monotone) rate of divergence $r$ can be computably converted to one that is monotone by setting
\begin{equation*}
\tilde r(n,x):=\max\{r(k,x) \, | \, k\leq n\}
\end{equation*}
Then $\tilde r$ is monotone in its first argument, to see that it is also a rate of divergence, note that $\tilde r(n,x)=r(k,x)\geq r(n,x)$ for some $k\leq n$, and so
\begin{equation*}
\sum_{i=n}^{\tilde r(n,x)}b_i\geq \sum_{i=n}^{r(n,x)}b_i\geq x
\end{equation*}
\end{remark}

%%%%%%%%%%%%%%%%%%%%%%%%%%%%%%%%%%%%%%%%%%%%%%%%%%%%%%%%%%%%%%%%%%%%%%%%%%%%%%%%%%%%%%%%%%
%%%%%%%%%%%%%%%%%%%%%%%%%%%%%%%%%%%%%%%%%%%%%%%%%%%%%%%%%%%%%%%%%%%%%%%%%%%%%%%%%%%%%%%%%%
\section{Abstract quantitative results}
\label{sec:recineq}
%%%%%%%%%%%%%%%%%%%%%%%%%%%%%%%%%%%%%%%%%%%%%%%%%%%%%%%%%%%%%%%%%%%%%%%%%%%%%%%%%%%%%%%%%%
%%%%%%%%%%%%%%%%%%%%%%%%%%%%%%%%%%%%%%%%%%%%%%%%%%%%%%%%%%%%%%%%%%%%%%%%%%%%%%%%%%%%%%%%%%

In this first main section, we restrict our attention to convergence properties of sequences of reals satisfying our main recursive inequality. To recall: for the remainder of the section, $\seq{\mu_n}$, $\seq{\alpha_n}$, $\seq{\beta_n}$ and $\seq{\gamma_n}$ will be sequences of nonnegative real numbers such that
\begin{equation}
\label{eqn:recineq:basic}\tag{$\star$}
\mu_{n+1}\leq \mu_n-\alpha_n\beta_n+\gamma_n
\end{equation}
for all $n\in\NN$. Moreover, we will always assume that 
\begin{itemize}

\item $\sum_{i=0}^\infty \alpha_i=\infty$, and

\item $\gamma_n\to 0$ as $n\to\infty$.

\end{itemize}
Our main question will concern the convergence behaviour of $\seq{\mu_n}$ and $\seq{\beta_n}$, and in particular under which conditions rates of convergence are possible. 

We begin with the observation that the above conditions alone do not guarantee convergence of either $\seq{\mu_n}$ or $\seq{\beta_n}$: After all, setting 
\begin{equation*}
\alpha_n=\frac{1}{n+1}, \ \ \beta_n=\sqrt{n+1}, \ \ \gamma_n=\frac{2}{\sqrt{n+1}}, \ \ \mu_0=0 \ \ \mbox{and} \ \ \mu_{n+1}=\sum_{i=0}^n \frac{1}{\sqrt{i+1}} 
\end{equation*}
we see that (\ref{eqn:recineq:basic}) is satisfied with $\sum_{i=0}^\infty\alpha_i=\infty$ and $\gamma_n\to 0$, but both $\seq{\mu_n}$ and $\seq{\beta_n}$ diverge. However, there are two key ways of strengthening $\gamma_n\to 0$ so that we start to see better behaviour, and these will in turn form the two main subclasses that will be studied below. The following facts are well-known:
\begin{proposition}[Alber and Iusem \cite{alber-iusem:01:subgradient}]
\label{prop:cases}
Suppose that $\seq{\mu_n}$, $\seq{\alpha_n}$, $\seq{\beta_n}$ and $\seq{\gamma_n}$ are sequences of nonnegative real numbers satisfying (\ref{eqn:recineq:basic}), and that $\sum_{i=0}^\infty \alpha_i=\infty$. Then either one of the following conditions implies the existence of a subsequence $\seq{\beta_{l_n}}$ with $\beta_{l_n}\to 0$:
\begin{enumerate}

	\item[(I)] $\sum_{i=0}^\infty \gamma_i<\infty$
	
	\item[(II)] $\gamma_n/\alpha_n\to 0$ as $n\to\infty$.

\end{enumerate}
Moreover, in the case (I) we also have $\mu_n\to \mu$ for some limit $\mu\geq 0$.
\end{proposition}

\begin{proof} %checked
We suppose that there exists $\varepsilon>0$ and $N\in\NN$ such that $\beta_n>\varepsilon$ for all $n\geq N$, and show that each of (I) and (II) lead to a contradiction, from which we can infer the existence of a subsequence $\seq{\beta_{l_n}}$ converging to zero.
\begin{enumerate}

	\item[(I)] Rearranging (\ref{eqn:recineq:basic}) we obtain
	\begin{equation*}
	\alpha_n\varepsilon<\alpha_n\beta_n\leq \mu_n-\mu_{n+1}+\gamma_n
	\end{equation*}
	for all $n\geq N$, and thus
	\begin{equation*}
	\varepsilon\sum_{i=N}^n \alpha_i\leq \mu_N-\mu_{n+1}+\sum_{i=N}^n \gamma_i\leq \mu_N+\sum_{i=0}^\infty \gamma_i
	\end{equation*}
	From $\sum_{i=0}^\infty \gamma_i<\infty$ we therefore infer $\sum_{i=0}^\infty \alpha_i<\infty$, a contradiction.
	
	\item[(II)] Pick $N_0\geq N$ such that $\gamma_n/\alpha_n\leq \varepsilon/2$ for all $n\geq N_0$. Then from (\ref{eqn:recineq:basic}) and $\beta_n>\varepsilon$ we have
	\begin{equation*}
	\mu_{n+1}< \mu_n-\alpha_n\varepsilon +\frac{\alpha_n\varepsilon}{2}=\mu_n-\frac{\alpha_n\varepsilon}{2}
	\end{equation*} 
	for all $n\geq N_0$, and therefore rearranging:
	\begin{equation*}
	\frac{\varepsilon}{2}\sum_{i=N_0}^n \alpha_i\leq \mu_{N_0}-\mu_{n+1}\leq \mu_{N_0}
	\end{equation*}
	and so again we have $\sum_{i=0}^\infty \alpha_i<\infty$, a contradiction.

\end{enumerate}
The additional fact that $\seq{\mu_n}$ converges to some limit in case (I) follows from the general result that whenever 
\begin{equation*}
\mu_{n+1}\leq \mu_n+\gamma_n
\end{equation*}
for $\sum_{i=0}^\infty \gamma_i<\infty$, then $\seq{\mu_n}$ converges, noting that since $\alpha_n\beta_n\geq 0$ then
\begin{equation*}
\mu_{n+1}\leq \mu_n-\alpha_n\beta_n+\gamma_n\leq \mu_n+\gamma_n
\end{equation*}
We omit details here, but the proof can be found in e.g. Alber and Guerre-Delabriere \cite[Proposition 1]{alber-guerredelabriere:01:projection}.
\end{proof}
The two cases (I) and (II) outlined above will be considered in Sections \ref{sec:recineq:parti} and \ref{sec:recineq:partii} respectively, where in each case further conditions are added to establish stronger convergence results. Before proceeding, we observe that the two cases are disjoint.
\begin{example} %checked
\begin{enumerate}

	\item[(a)] Define $\alpha_n=1/(n+1)$, $\beta_n=0$ and $\mu_{n+1}=\mu_n+\gamma_n$ for some $\mu_0\geq 0$, where
	\begin{equation*}
	\gamma_n=\begin{cases}
		1/(n+1) & \mbox{if $n+1=m^2$ for some $m\in\NN$}\\
		0 & \mbox{otherwise}
	\end{cases}
	\end{equation*}
	Then these sequences satisfy case (I) of Proposition \ref{prop:cases} but not case (II), since
	\begin{equation*}
	\sum_{i=0}^\infty \gamma_i= \sum_{i=1}^\infty \frac{1}{i^2}<\infty
	\end{equation*}
	but whenever $n+1=m^2$ we have $\gamma_n/\alpha_n=1$, therefore we cannot have $\gamma_n/\alpha_n\to 0$.

	\item[(b)] Define $\alpha_n=\beta_n=1/\sqrt{n+1}$, $\gamma_n=1/(n+1)$ and $\mu_n=\mu_0$ for some $\mu_0\geq 0$. Then these sequences satisfy case (II) of Proposition \ref{prop:cases} but not case (I).

\end{enumerate}
\end{example}

%%%%%%%%%%%%%%%%%%%%%%%%%%%%%%%%%%%%%%%%%%%%%%%%%%%%%%%%%%%%%%%%%%%%%%%%%%%%%%%%%%%%%%%%%%
%%%%%%%%%%%%%%%%%%%%%%%%%%%%%%%%%%%%%%%%%%%%%%%%%%%%%%%%%%%%%%%%%%%%%%%%%%%%%%%%%%%%%%%%%%
\subsection{Case I: Convergence with $\sum_{i=0}^\infty \gamma_i<\infty$}
\label{sec:recineq:parti}
%%%%%%%%%%%%%%%%%%%%%%%%%%%%%%%%%%%%%%%%%%%%%%%%%%%%%%%%%%%%%%%%%%%%%%%%%%%%%%%%%%%%%%%%%%
%%%%%%%%%%%%%%%%%%%%%%%%%%%%%%%%%%%%%%%%%%%%%%%%%%%%%%%%%%%%%%%%%%%%%%%%%%%%%%%%%%%%%%%%%%

As stated in Proposition \ref{prop:cases} above, if $\sum_{i=0}^\infty \gamma_i<\infty$, then $\seq{\mu_n}$ converges to some limit independently of any properties of $\seq{\alpha_n}$ and $\seq{\beta_n}$. Indeed, more general results of this kind are possible: For instance, Qihou \cite{qihou:01:recineq} proves that if
%Slight repetition of above
%
\begin{equation}
\label{eqn:qihou}
\mu_{n+1}\leq (1+\delta_n)\mu_n+\gamma_n
\end{equation}
where $\seq{\delta_n}$ is also sequence of nonnegative reals with $\sum_{i=0}^\infty \delta_n < \infty$, then $\seq{\mu_n}$ converges. 
%Even in the simple case that $\gamma_n=\delta_n=0$, the existence of Specker sequences in the sense of Proposition \ref{prop:specker} means that computable rates of convergence are not possible. 
Convergence of $\seq{\beta_n}$, on the other hand, is more subtle, and requires additional assumptions beyond $\sum_{i=0}^\infty \gamma_i<\infty$, as the following example demonstrates:
\begin{example} %checked 
Let $\alpha_n:=1/(n+1)$, $\gamma_n:=0$ and $\beta_n:=1$ if $n+1=m^2$ else $0$. Then
\begin{equation*}
    \sum_{i=0}^\infty \alpha_i\beta_i=\sum_{i=m^2} \frac{1}{i}=\sum_{i=0}^\infty \frac{1}{i^2}\leq L
\end{equation*}
for sufficiently large $L$. Now let $\mu_0:=L$ and $\mu_{n+1}:=L-\sum_{i=0}^n \alpha_i\beta_i$. Then $\mu_n\geq 0$ for all $n\in\NN$ and $\mu_{n+1}\leq \mu_n-\alpha_n\beta_n$ with $\sum_{i=0}^\infty \alpha_i=\infty$, but $\seq{\beta_n}$ does not converge.
\end{example}
A key result in this direction is due to Alber and Iusem., which adds a regularity condition to the $\seq{\beta_n}$ to establish convergence to zero:
\begin{proposition}
[cf. Lemma 2.2 of Alber and Iusem \cite{alber-iusem:01:subgradient}]
\label{prop:recineq1:conv}
Suppose that $\seq{\mu_n}$, $\seq{\alpha_n}$, $\seq{\beta_n}$ and $\seq{\gamma_n}$ are sequences of nonnegative real numbers satisfying (\ref{eqn:recineq:basic}) with $\sum_{i=0}^\infty\alpha_i=\infty$ and $\sum_{i=0}^\infty \gamma_i<\infty$. Then whenever there exists $\theta>0$ such that the following conditions holds:
\begin{equation*}
\beta_n-\beta_{n+1}\leq \theta\alpha_n \ \ \mbox{for all} \ \ n\in\NN
\end{equation*}
then $\mu_n\to \mu$ for some $\mu\geq 0$, and $\beta_n\to 0$.
\end{proposition}
The above result and its variants, proofs of which are typically by contradiction (note that an alternative proof of the result is presented within this paper, via Propositions \ref{prop:recineq1:beta:conv} and \ref{prop:recineq1:beta:generalise}), plays a particularly important role in establishing convergence properties of subgradient algorithms (cf. \cite{alber-iusem:01:subgradient,alber-iusem-solodov:98:subgradient,drummond-iusem:04:projected:gradient,mainge:08:projected:subgradient} as examples), and we will explore this domain of application later in Section \ref{sec:applications:gradient}. Therefore the question of whether computable rates of convergence for $\mu_n\to \mu$ and particularly $\beta_n\to 0$ can be produced is important, as these could be used to obtain computable rates of convergence for a very general class of subgradient methods. Our first main result gives a negative answer, and we also conjecture that the main construction here could be easily adapted to prove similar results for the numerous variants on Proposition \ref{prop:recineq1:conv} found in the literature. We first need a simple lemma which adapts the standard construction given in Proposition \ref{prop:specker:zero}.
\begin{lemma}
\label{cor:speckersum}
There exists a computable sequence $\seq{s_n}$ of rational numbers with $\sum_{i=0}^\infty s_i<\infty$ such that $s_n\to 0$ with no computable rate of convergence. 
\end{lemma}

\begin{proof} %checked
Let $\seq{a_n}$ and $\seq{s_n}$ be as in Proposition \ref{prop:specker:zero} but with $\sum_{i=0}^\infty a_i<\infty$ (e.g. $a_n=1/(n+1)^2$). Define $b_n:=a_n$ if there exists some $k\in\NN$ such that $s_k=a_n$, and $0$ otherwise, and note that $\sum_{i=0}^\infty b_i\leq \sum_{i=0}^\infty a_i<\infty$. Let $\seq{s_{n_i}}$ (respectively $\seq{b_{m_j}}$), be the subsequence of $\seq{s_n}$ (respectively $\seq{b_n}$) consisting of the sequence's nonzero elements. Then there is a bijection $f:\NN\to\NN$ such that  $s_{n_i}=b_{m_{f(i)}}$ for all $i\in\NN$. To define $f$, referring to the definition of $\seq{s_n}$, for $i\in\NN$ let $p$ to be the least index $\leq n_i$ such that $T_{p}(p)$ halts in $n_i$ steps, which we know to exist since $s_{n_i}$ is nonzero, and so in fact $s_{n_i}=a_{p}$. This therefore means that $b_{p}=a_{p}>0$, and so $b_{p}=b_{m_j}$ for some index $j$, and we define $f(i):=j$ for this index. 

This map is surjective since for any $j$, $b_{m_j}>0$ and thus $b_{m_j}=a_{m_j}$ where there exists some $k$ with $s_k=a_{m_j}>0$, and so in turn $k=n_i$ for some unique index $i$, and going back through the definition of $f$ we see that $f(i)=j$. It is injective because if $s_{n_i}=s_{n_{i'}}=b_{m_j}$, then $T_{m_j}(m_j)$ halts in exactly $n_i$ and $n_{i'}$ steps, and thus $i=i'$. The existence of $f$ means that $\seq{s_{n_i}}$ is just a permutation of $\seq{b_{m_j}}$, and therefore
\begin{equation*}
\sum_{i=0}^\infty s_i=\sum_{i=0}^\infty s_{n_i}=\sum_{j=0}^\infty b_{m_j}=\sum_{j=0}^\infty b_j <\infty
\end{equation*}
where the first equality follows since $\seq{s_i}$ is just $\seq{s_{n_i}}$ padded out with zero elements (and similarly for the third), while the second equality follows from the fact that we can reorder the terms in series where all terms are positive.
\end{proof} 
\begin{theorem}
\label{thm:block}
For any sequence of positive reals $\seq{\alpha_n}$ bounded above by $\alpha>0$ with $\sum_{i=0}^\infty\alpha_i=\infty$, together with $\theta>0$, there exist sequences of positive reals $\seq{\mu_n}$ and $\seq{\beta_n}$, computable in $\seq{\alpha_n}$ and $\theta$ and satisfying
\begin{equation*}
\mu_{n+1}\leq \mu_n-\alpha_n\beta_n \ \ \ \mbox{and} \ \ \ \beta_n-\beta_{n+1}\leq \theta\alpha_n
\end{equation*}
such that $\mu_n\to \mu$ and $\beta_n\to 0$, but neither with a computable rate of convergence. %Moreover, if $\seq{\alpha_n}$ is a computable sequence of rational numbers, and $\alpha$ and $\theta$ both rational, $\seq{\mu_n}$ and $\seq{\beta_n}$ are both computable sequences of rational numbers.[Keji comment: If $\alpha_n$ unbounded, construct an new sequence $\alpha_n'$ such that $\alpha_n'$ is bounded and for all n $ \alpha_n' \le \alpha_n$. Thus, applying the construction for $\beta_n$ with $alpha_n'$ will produce a sequence that satisfies the required properties for $\alpha_n$ and to and we construct $\mu_n$ in exactly the same way (with $\alpha_n$)] 
\end{theorem}

\begin{proof}%checked
Take $\seq{s_n}$ as in Proposition \ref{cor:speckersum} and define $l:\NN\to\NN$ recursively by $l(0):=0$, and $l(n+1):=k+1$ where $k\geq l(n)$ is the least number satisfying
\begin{equation}
\label{eqn:propk}
\sum_{i=l(n)}^{k} \alpha_i\geq \frac{|s_{n}-s_{n+1}|}{\theta}
\end{equation}
which is well defined by $\sum_{i=0}^\infty \alpha_i=\infty$. Now define $\seq{\beta_k}$ as follows: $\beta_{l(n)}:=s_n$, and if $l(n)<k<l(n+1)$ then
\begin{equation*}
\beta_k:=s_n+\sgn(s_{n+1}-s_n)\cdot\theta\sum_{i=l(n)}^{k-1}\alpha_i
\end{equation*}
The intuition here is that $\seq{\beta_n}$ represents a `padding out' of $\seq{s_n}$, designed in such as way as to ensure that the distance between consecutive elements is bounded by $\theta\alpha_n$. Since $l(n)$ is strictly increasing with $l(0)=0$, $\seq{\beta_k}$ is thereby defined for all $k\in\NN$. We now show that $|\beta_k-\beta_{k+1}|\leq \theta\alpha_k$. There are two cases to deal with: If $l(n)\leq k<l(n+1)-1$ then 
%$|s_n-s_{n+1}|>0$ (else we would have $l(n+1)=l(n)+1$ since $\seq{\alpha_i}$ are positive), and therefore
%
\begin{equation*}
|\beta_k-\beta_{k+1}|=|\sgn(s_{n+1}-s_n)\cdot\theta\alpha_k| \leq \theta\alpha_k
\end{equation*} 
and if $k=l(n+1)-1$ then (assuming w.l.o.g. that $|s_n-s_{n+1}|>0$, otherwise the claim is trivially true):
\begin{equation*}
\begin{aligned}
|\beta_{l(n+1)-1}-\beta_{l(n+1)}|&=\left|s_n+\sgn(s_{n+1}-s_n)\cdot\theta\sum_{i=l(n)}^{l(n+1)-2}\alpha_i-s_{n+1}\right|\\
&=|\sgn(s_{n+1}-s_n)|\cdot \left|\theta\sum_{i=l(n)}^{l(n+1)-2}\alpha_i-|s_{n}-s_{n+1}|\right|\\
&=|s_{n}-s_{n+1}|-\theta\sum_{i=l(n)}^{l(n+1)-2}\alpha_i\\
&\leq  \theta\sum_{i=l(n)}^{l(n+1)-1}\alpha_i-\theta\sum_{i=l(n)}^{l(n+1)-2}\alpha_i\\
&=\theta\alpha_{l(n+1)-1}
\end{aligned}
\end{equation*}
where if $l(n+1)-1=l(n)$ we use the convention that $\sum_{i=l(n)}^{l(n+1)-2}\alpha_i=0$. For the third equality we use that $|\sgn(s_n-s_{n+1})|=1$ and $\theta\sum_{i={l(n)}}^{l(n+1)-2}\alpha_i\leq |s_n-s_{n+1}|$, since otherwise this would contradict the construction of $l(n+1)$. 

From the defining property of $l(n+1)$ we can also show that for $l(n)\leq k<l(n+1)$ we have either $s_n\leq \beta_k\leq s_{n+1}$ or $s_{n+1}\leq \beta_k\leq s_n$, and from this and the fact that $s_n\to 0$ it follows that $\beta_n\to 0$.

To show that there is no computable rate of convergence for $\beta_n\to 0$, suppose for contradiction that $\phi$ is such a rate. Fixing $\varepsilon>0$, we have $\beta_n\leq \varepsilon$ for all $n\geq \phi(\varepsilon)$. However, since $l(n)$ is strictly monotone, we have $l(n)\geq n$ for all $n\in\NN$, and thus $s_n=\beta_{l(n)}\leq \varepsilon$ for all $n\geq \phi(\varepsilon)$. Therefore $\phi$ is also a computable rate of convergence for $s_n\to 0$, which is not possible.

Finally, to define $\seq{\mu_n}$, we first show that $\sum_{i=0}^\infty\alpha_i\beta_i<\infty$. Let $c>0$ be any upper bound on $\seq{s_n}$. Using again that $s_n\leq \beta_k\leq s_{n+1}$ or $s_{n+1}\leq \beta_k\leq s_n$ for $l(n)\leq k<l(n+1)$, we have $\beta_k\leq s_n+s_{n+1}$ for $k$ in this range, and therefore
\begin{equation*}
\begin{aligned}
\sum_{i=l(n)}^{l(n+1)-1}\alpha_i\beta_i&\leq (s_n+s_{n+1})\sum_{i=l(n)}^{l(n+1)-1}\alpha_i\\
&\leq(s_n+s_{n+1})\left(\frac{|s_n-s_{n+1}|}{\theta}+\alpha_{l(n+1)-1}\right)\\
&\leq (s_n+s_{n+1})(c/\theta+\alpha)
\end{aligned}
\end{equation*}
where for the second inequality we use the defining property of $l(n+1)$, noting that since $l(n+1)-1=k$ for the \emph{least} $k$ satisfying \ref{eqn:propk}, we must have
\begin{equation*}
\sum_{i=l(n)}^{l(n+1)-1}\alpha_i=\sum_{i=l(n)}^{k}\alpha_i\leq \frac{|s_n-s_{n+1}|}{\theta}+\alpha_k
\end{equation*}
since if this were not the case then
\begin{equation*}
\sum_{i=l(n)}^{k-1}\alpha_i>\frac{|s_n-s_{n+1}|}{\theta}
\end{equation*}
for $k-1\geq l(n)$, contradicting minimality of $k$. Therefore summing over the whole sequence:
\begin{equation*}
\begin{aligned}
\sum_{i=0}^\infty \alpha_i\beta_i&=\sum_{n=0}^\infty \sum_{i=l(n)}^{l(n+1)-1} \alpha_i\beta_i\\
&\leq(c/\theta+\alpha)\sum_{n=0}^\infty (s_n+s_{n+1})\\
&\leq 2(c/\theta+\alpha)\sum_{n=0}^\infty s_n<\infty
\end{aligned}
\end{equation*}
Now to construct $\seq{\mu_n}$, first of all, let $L>0$ be any bound with $\sum_{i=0}^\infty \alpha_i\beta_i\leq L$, and let $\seq{\nu_n}$ be some computable monotonically decreasing sequence converging to some limit $\nu\geq 0$. Define
\begin{equation*}
\mu_0:= L +\nu_0 \ \ \ \mbox{and} \ \ \ \mu_{n+1}:=L+\nu_{n+1}-\sum_{i=0}^n \alpha_i\beta_i
\end{equation*}
Then $\mu_{n+1}\leq \mu_n-\alpha_n\beta_n$, and we now claim that if exactly one of $\seq{\nu_n}$ and $\sum_{i=0}^\infty \alpha_i\beta_i$ have a computable rate of Cauchy convergence, then $\seq{\mu_n}$ cannot have a computable rate of Cauchy convergence. To see this we observe that for any $m\geq n$:
\begin{equation*}
\mu_m-\mu_n= (\nu_m-\nu_n)-\sum_{i=n-1}^{m-1}\alpha_i\beta_i
\end{equation*}
and so if $\seq{\mu_n}$ has a computable rate of Cauchy convergence, then a computable rate of convergence for either one of $\seq{\nu_n}$ or $\sum_{i=0}^\infty \alpha_i\beta_i$ would imply a computable rate for the other.

It remains to show that we can always define $\seq{\nu_n}$ such that exactly one of $\seq{\nu_n}$ or $\sum_{i=0}^\infty \alpha_i\beta_i$ have a computable rate of convergence: If $\sum_{i=0}^\infty \alpha_i\beta_i$ has a computable rate, then let $\seq{\nu_n}$ be some computable, monotonically decreasing sequence of rationals that converges to some noncomputable $\nu>0$. Note that this exists by adapting Proposition \ref{prop:specker}: Taking $\seq{a_n}$ to be as in that proposition i.e. a computable, monotonically increasing and with a noncomputable limit $a$, and $c>0$ some rational upper bound for $\seq{a_n}$, we can define $\nu_n:=c-a_n$ so that $\seq{\nu_n}$ is then monotonically decreasing and converges to $c-a$. Since $c$ is computable and $a$ is not, $c-a$ must also be noncomputable. On the other hand, if $\sum_{i=0}^\infty \alpha_i\beta_i$ does not have a computable rate, just set $\nu_n=0$ for all $n\in\NN$. 

The final claim that $\seq{\mu_n}$ and $\seq{\beta_n}$ are computable sequences of rationals when $\seq{\alpha_n}$ is followed by observing that all of the above constructions are computable, including the construction $\seq{\nu_n}$ above, along with $\seq{s_n}$ as taken from Proposition \ref{prop:specker:zero}. We note that our proof gives two possible constructions for $\seq{\mu_n}$, both of which are computable in $\seq{\alpha_n}$ and $\theta$, according to whether or not $\sum_{i=0}^\infty\alpha_i\beta_i<\infty$ with a computable rate of convergence, and thus the theorem holds, even if we cannot read off from the proof which of these two constructions we would use. 
\end{proof}
As a consequence of Theorem \ref{thm:block}, we cannot in general obtain rates of convergence for $\mu_n\to \mu$ and $\beta_n\to 0$ in Proposition \ref{prop:recineq1:conv}, since even in the special case that $\gamma_n=0$ and $\seq{\alpha_n}$ is computable such that $\sum_{i=0}^\infty \alpha_i=\infty$ with a computable rate of divergence (e.g. $\alpha_n=1$), there exist computable $\seq{\mu_n}$ and $\seq{\beta_n}$ such that all of the conditions of the proposition are satisfied, but neither $\seq{\mu_n}$ nor $\seq{\beta_n}$ converge with a computable rate. In the special case that $\alpha_n:=1$, we have $\sum_{i=0}^\infty \alpha_i\beta_i=\sum_{i=0}^\infty \beta_i<\infty$, and a rate of convergence for the latter would also imply a rate of convergence for $\beta_n\to 0$, which is not possible, therefore in this case we can also decide which of the two constructions for $\seq{\mu_n}$ we would use, defining $\mu_n:=L-\sum_{i=0}^n\beta_i$.

Having shown that computable rates of convergence are not possible in general, we can nevertheless give a quantitative version of Proposition \ref{prop:recineq1:conv} in terms of rates of metastability. We first observe that quantitative versions of recursive inequalities of the form (\ref{eqn:qihou}) have been widely used in the applied proof theory literature, and rates of metastability for $\seq{\mu_n}$ which are given by Kohlenbach and Lambov \cite{kohlenbach-lambov:04:asymptotically:nonexpansive} would apply directly to the case of Proposition \ref{prop:recineq1:conv}. As such, we focus here on completing the picture by giving a quantitative version of the property that $\beta_n\to 0$. In this case, the core of Proposition \ref{prop:recineq1:conv} is contained in the following result:
\begin{proposition}[cf. Proposition 2 of Alber, Iusem and Solodov \cite{alber-iusem-solodov:98:subgradient}]
\label{prop:recineq1:beta:conv}
Suppose that $\seq{\alpha_n}$ and $\seq{\beta_n}$ are sequences of nonnegative real numbers with $\sum_{i=0}^\infty\alpha_i=\infty$ and $\sum_{i=0}^\infty \alpha_i\beta_i<\infty$. Then whenever there exists $\theta>0$ such that the following condition holds:
\begin{equation*}
\beta_n-\beta_{n+1}\leq \theta\alpha_n \ \ \mbox{for all} \ \ n\in\NN
\end{equation*}
then $\beta_n\to 0$.
\end{proposition}
To derive $\beta_n\to 0$ under the conditions of Proposition \ref{prop:recineq1:conv}, we simply observe that by (\ref{eqn:recineq:basic}) and condition (i) of Proposition \ref{prop:recineq1:conv} we have
\begin{equation*}
\sum_{i=0}^n\alpha_i\beta_i\leq \mu_0-\mu_{n+1}+\sum_{i=0}^n\gamma_i\leq \mu_0+\sum_{i=0}^\infty\gamma_i
\end{equation*}
and thus $\sum_{i=0}^\infty\alpha_i\beta_i<\infty$, and so Proposition \ref{prop:recineq1:beta:conv} applies directly. 

The original proof of Proposition \ref{prop:recineq1:beta:conv} uses a somewhat intricate and highly ineffective subsequence construction (cf. Proposition 2 of \cite{alber-iusem-solodov:98:subgradient}). We now provide a strengthening of the result (and hence also of Proposition \ref{prop:recineq1:conv}) which establishes a necessary and sufficient condition for $\beta_n\to 0$ under $\sum_{i=0}^\infty \alpha_i\beta_i < \infty$, or in the case of Proposition \ref{prop:recineq1:beta:conv}, under$\sum_{i=0}^\infty \gamma_i<\infty$. We then follow this with a quantitative variant, giving a simple constructive formulation of the standard proof strategy exhibited in \cite{alber-iusem-solodov:98:subgradient}.
\begin{proposition}
\label{prop:recineq1:beta:generalise}
Suppose that $\seq{\alpha_n}$ and $\seq{\beta_n}$ are sequences of nonnegative real numbers with $\sum_{i=0}^\infty\alpha_i=\infty$ and $\sum_{i=0}^\infty \alpha_i\beta_i<\infty$. Then $\beta_n\to 0$ if and only if there exists $\theta>0$ such that
\begin{equation*}
\limsup_{N\to\infty}\left\{\beta_n-\beta_{m}-\theta\sum_{i=n}^{m-1}\alpha_i \, | \, N\leq n< m\right\}\leq 0
\end{equation*}
\end{proposition}

We initially leave Proposition \ref{prop:recineq1:beta:generalise} without proof, as one direction will follow from the quantitative version of the result below. 
\begin{theorem}
[Quantitative version of Proposition \ref{prop:recineq1:beta:generalise}]
\label{thm:recineq1:beta:comp}
Suppose that $\seq{\alpha_n}$ and $\seq{\beta_n}$ are sequences of nonnegative real numbers and $r$ is a rate of divergence for $\sum_{i=0}^\infty\alpha_i=\infty$. Let $\varepsilon>0$ and $g:\NN\to\NN$ be arbitrary. Then if $N_1,N_2\in\NN$ and $\theta>0$ are such that, setting $N:=\max\{N_1,N_2\}$, we have
\begin{equation*}
\sum_{i=N_1}^{r(N+g(N),\varepsilon/4\theta)} \alpha_i\beta_i\leq \frac{\varepsilon^2}{8\theta} 
\end{equation*}
and
\begin{equation*}
\beta_n-\beta_{m}\leq \theta\sum_{i=n}^{m-1}\alpha_i+\frac{\varepsilon}{4} \ \ \mbox{for all }N_2\leq n<m\leq  r\left(N+g(N),\frac{\varepsilon}{4\theta}\right)
\end{equation*} 
then we can conclude that $\beta_n\leq \varepsilon$ for all $n\in [N,N+g(N)]$.
\end{theorem}

\begin{proof}%checked
Fix $\varepsilon>0$ and $g:\NN\to\NN$, and suppose for contradiction that $\beta_n>\varepsilon$ for some $n\in [N,N+g(N)]$. We first show that there exists some $m\in [n,r(n,\varepsilon/4\theta)]$ with $\beta_m\leq \varepsilon/2$: If this were not the case then using monotonicity of $r$ in its first component (cf. Remark \ref{rem:rod}) we would have
\begin{equation*}
\frac{\varepsilon}{4\theta}\leq\sum_{i=n}^{r(n,\varepsilon/4\theta)}\alpha_i<\frac{2}{\varepsilon}\sum_{i=n}^{r(n,\varepsilon/4\theta)} \alpha_i\beta_i\leq \frac{2}{\varepsilon}\sum_{i=N_1}^{r(N+g(N),\varepsilon/4\theta)} \alpha_i\beta_i\leq \frac{\varepsilon}{4\theta}
\end{equation*}
which contradicts the claim, where for the third inequality we use that $N_1\leq N\leq n$, and $n\leq N+g(N)$ together with our assumption that rates of divergence $r$ are monotone in their first argument. Now let $n<m\leq r(n,\varepsilon/4\theta)$ be the least such index such that $\beta_n>\varepsilon$, $\beta_i>\varepsilon/2$ for $i = n,\ldots, m-1$ and $\beta_m\leq \varepsilon/2$. Then since $N_2\leq N\leq n$ we have
\begin{equation*}
\begin{aligned}
\frac{\varepsilon}{2}<\beta_n-\beta_m&\leq \theta\sum_{i=n}^{m-1}\alpha_i+\frac{\varepsilon}{4} \leq \frac{2\theta}{\varepsilon}\sum_{i=n}^{m-1}\alpha_i\beta_i+\frac{\varepsilon}{4}\\
&\leq \frac{2\theta}{\varepsilon}\sum_{i=N_1}^{r(N+g(N),\varepsilon/4\theta)}\alpha_i\beta_i+\frac{\varepsilon}{4}\leq \frac{\varepsilon}{2}
\end{aligned}
\end{equation*}
and so we have our contradiction.
\end{proof}

\begin{proof}[Proof of Proposition \ref{prop:recineq1:beta:generalise}]%checked
To establish one direction we use Theorem \ref{thm:recineq1:beta:comp}. If $\sum_{i=0}^\infty\alpha_i=\infty$ then it possesses some rate of divergence, and for any $\varepsilon>0$ there must exist some $N_1$ and $N_2$ such that
\begin{equation*}
\sum_{i=N_1}^\infty \alpha_i\beta_i\leq \frac{\varepsilon^2}{8\theta}
\end{equation*}
and
\begin{equation*}
\beta_n-\beta_m-\theta\sum_{i=n}^{m-1}\alpha_i\leq \frac{\varepsilon}{4}
\end{equation*}
for all $N_2\leq n<m$. Therefore the assumptions of Proposition \ref{thm:recineq1:beta:comp} are satisfied for arbitrary $g:\NN\to\NN$, thus $\beta_n\leq \varepsilon$ for all $n\in [N, N+g(N)]$ also for arbitrary $g$. But this just means that $\beta_n\leq \varepsilon$ for all $n\geq N$, and so $\beta_n\to 0$.

In the other direction, if $\beta_n\to 0$ then there exists $N\in\NN$ such that $|\beta_n-\beta_m|\leq \varepsilon$ for all $m,n\geq N$. But then for $m>n\geq N$ we have
\begin{equation*}
\beta_n-\beta_m\leq |\beta_n-\beta_m|\leq \varepsilon\leq \sum_{i=n}^{m-1}\alpha_i+\varepsilon
\end{equation*}
and so $\limsup_{N\to\infty}\left\{\beta_n-\beta_{m}-\sum_{i=n}^{m-1}\alpha_i \, | \, N\leq n< m\right\}\leq 0$.
\end{proof}
Proposition \ref{prop:recineq1:beta:conv} follows directly from Proposition \ref{prop:recineq1:beta:generalise} because if $\beta_n-\beta_{n+1}\leq \theta\alpha_n$ for all $n\in\NN$, we have
\begin{equation*}
\beta_n-\beta_m- \theta\sum_{i=n}^m\alpha_i\leq 0
\end{equation*}
for any $n<m$. In particular, we can set $N_2:=0$ in Theorem \ref{thm:recineq1:beta:comp} and replace $\varepsilon^2/8\theta$ with $\varepsilon^2/4\theta$ for the property of $N_1$, so that Theorem \ref{thm:recineq1:beta:comp} simplifies to the following:
\begin{corollary}
\label{cor:simplified}
Suppose that $\seq{\alpha_n}$ and $\seq{\beta_n}$ are sequence of nonnegative real numbers and $r$ is a rate of divergence for $\sum_{i=0}^\infty\alpha_i=\infty$. Let $\varepsilon>0$ and $g:\NN\to\NN$ be arbitrary. Then if $N\in\NN$ and $\theta>0$ are such that
\begin{equation*}
\sum_{i=N}^{r(N+g(N),\varepsilon/2\theta)}\alpha_i\beta_i\leq \frac{\varepsilon^2}{4\theta} \ \ \ \mbox{and} \ \ \ \beta_n-\beta_{n+1}\leq \theta\alpha_n \ \ \ \mbox{for all $n\in\NN$}
\end{equation*}
then $\beta_n\leq \varepsilon$ for all $n\in [N,N+g(N)]$.
\end{corollary}
As an immediate corollary we obtain the following:
\begin{corollary}%checked
\label{cor:recineq1:beta:rates}
Suppose that $\seq{\alpha_n}$ and $\seq{\beta_n}$ are sequences of nonnegative real numbers and $r$ is a rate of divergence for $\sum_{i=0}^\infty\alpha_i=\infty$, and that there is some $\theta>0$ such that $\beta_n-\beta_{n+1}\leq \theta\alpha_n$ for all $n\in\NN$. Then
\begin{enumerate}[(a)]

	\item If $\sum_{i=0}^\infty \alpha_i\beta_i<\infty$ with rate of convergence $\phi$, then $\beta_n\to 0$ with rate of convergence
	\begin{equation*}
		\psi(\varepsilon):=\phi\left(\frac{\varepsilon^2}{4\theta}\right)
	\end{equation*}
	
	\item If $\sum_{i=0}^\infty \alpha_i\beta_i<\infty$ with rate of metastability $\Phi$, then $\beta_n\to 0$ with rate of metastability
	\begin{equation*}
		\Psi(\varepsilon,g):=\Phi\left(\frac{\varepsilon^2}{4\theta},h\right) \ \ \mbox{for} \ \ h(n):=r\left(n+g(n),\frac{\varepsilon}{2\theta}\right)-n
	\end{equation*}

\end{enumerate}

\end{corollary}

\begin{proof}
\begin{enumerate}[(a)]

	\item Since $\phi$ is a rate of convergence for $\sum_{i=0}^\infty \alpha_i\beta_i<\infty$, for any $N\geq \psi(\epsilon)$ we have $\sum_{i=N}^{N+m}\alpha_i\beta_i\leq\varepsilon^2/4\theta$ for any $m\in\NN$, so setting $g(i):=0$ in Corollary \ref{cor:simplified} we have
	\begin{equation*}
	\sum_{i=N}^{r(N,\varepsilon/2\theta)}\alpha_i\beta_i\leq \frac{\varepsilon^2}{4\theta}
	\end{equation*}
	and therefore $\beta_n\leq \varepsilon$ for $n\in [N,N]$ i.e. $\beta_N\leq \varepsilon$.
	
	\item Similarly, if $\Phi$ is a rate of metastability for $\sum_{i=0}^\infty\alpha_i\beta_i<\infty$, there exists some $N\leq \Psi(\varepsilon,g)$ such that
	\begin{equation*}
	\sum_{i=N}^{N+h(N)}\alpha_i\beta_i=\sum_{i=N}^{r(N+g(N),\varepsilon/2\theta}\alpha_i\beta_i\leq\frac{\varepsilon^2}{4\theta}
	\end{equation*}
	and thus by Corollary \ref{cor:simplified} we have $\beta_n\leq\varepsilon$ for all $n\in [N,N+g(N)]$.

\end{enumerate}
\end{proof}

%%%%%%%%%%%%%%%%%%%%%%%%%%%%%%%%%%%%%%%%%%%%%%%%%%%%%%%%%%%%%%%%%%%%%%%%%%%%%%%%%%%%%%%%%%
%%%%%%%%%%%%%%%%%%%%%%%%%%%%%%%%%%%%%%%%%%%%%%%%%%%%%%%%%%%%%%%%%%%%%%%%%%%%%%%%%%%%%%%%%%
\subsection{Case II: Convergence with $\gamma_n/\alpha_n\to 0$}
\label{sec:recineq:partii}
%%%%%%%%%%%%%%%%%%%%%%%%%%%%%%%%%%%%%%%%%%%%%%%%%%%%%%%%%%%%%%%%%%%%%%%%%%%%%%%%%%%%%%%%%%
%%%%%%%%%%%%%%%%%%%%%%%%%%%%%%%%%%%%%%%%%%%%%%%%%%%%%%%%%%%%%%%%%%%%%%%%%%%%%%%%%%%%%%%%%%

Having thoroughly analysed (\ref{eqn:recineq:basic}) under the additional condition that $\sum_{i=0}^\infty \gamma_i<\infty$, we now move onto the second condition that allows us to obtain favourable convergence results on $\seq{\mu_n}$ and $\seq{\beta_n}$. This second case has a slightly different flavour: Here the condition on $\seq{\gamma_n}$ is primarily used to ensure convergence of $\seq{\mu_n}$ towards zero. The fact that $\seq{\beta_n}$ possesses a subsequence converging to zero is crucial for this, together with a requirement that there exists some modulus linking this subsequence to a corresponding subsequence of $\seq{\mu_n}$. We show that the existence of such a modulus is also necessary when our sequences are strictly positive. Our starting point is a standard convergence result from the literature.
\begin{proposition}
[cf. Lemma 2.5 of Alber and Iusem \cite{alber-iusem:01:subgradient}]
\label{prop:recineq2:conv}
Suppose that $\seq{\mu_n}$, $\seq{\alpha_n}$, $\seq{\beta_n}$ and $\seq{\gamma_n}$ are sequences of nonnegative real numbers with $\alpha_n>0$ for all $n\in\NN$, satisfying (\ref{eqn:recineq:basic}) with $\sum_{i=0}^\infty\alpha_i=\infty$, $\seq{\alpha_n}$ bounded and $\gamma_n/\alpha_n \to 0$ as $n\to\infty$. Then there exists some sequence of indices $\seq{l_n}$ such that $\beta_{l_n}\to 0$ as $n\to\infty$, and whenever, in addition, we have $\mu_{l_n}\to 0$, then $\mu_n\to 0$ as $n\to\infty$.
\end{proposition}
Typically, the above result is applied in cases where $\beta_n=\psi(\mu_n)$ for some continuous, nondecreasing function $\psi:[0,\infty)\to [0,\infty)$ which is strictly positive on $(0,\infty)$ and satisfies $\psi(0)=0$. These properties allow us to infer $\mu_{l_n}\to 0$ from $\psi(\mu_{l_n})\to 0$, and thus $\lim_{n\to\infty}\mu_n=0$. We now give an alternative formulation of Proposition \ref{prop:recineq2:conv} above, where we replace $\beta_n=\psi(\mu_n)$ with a more general modulus $\omega$. In the case that the elements $\seq{\beta_n}$ are strictly positive, the existence of such a modulus is even necessary for $\mu_n\to 0$, which justifies it as a natural property.
\begin{proposition}
\label{prop:recineq2}
Suppose that $\seq{\mu_n}$, $\seq{\alpha_n}$, $\seq{\beta_n}$ and $\seq{\gamma_n}$ are sequences of nonnegative real numbers with $\alpha_n>0$ for all $n\in\NN$, satisfying (\ref{eqn:recineq:basic}) with $\sum_{i=0}^\infty \alpha_i=\infty$, $\seq{\alpha_n}$ bounded and $\gamma_n/\alpha_n\to 0$. Then if there exists a function $\omega:\QQ_+\to \QQ_+$ such that for every $\varepsilon>0$ we have
\begin{equation}
\label{eqn:modulus}
\beta_n\leq \omega(\varepsilon)\implies \mu_n\leq \varepsilon
\end{equation}
for all $n\in\NN$, then $\mu_n\to 0$ as $n\to\infty$. Conversely, if $\mu_n\to 0$ and $\mu_n>0\implies \beta_n>0$ for all $n\in\NN$, then there necessarily exists an $\omega$ satisfying (\ref{eqn:modulus}) for all $\varepsilon>0$ and $n\in\NN$.
\end{proposition}

\begin{proof}%checked
One direction is by Proposition \ref{prop:recineq2:conv}. Here $\gamma_n/\alpha_n\to 0$ implies there exists some subsequence $\beta_{l_n}\to 0$ as outlined in Proposition \ref{prop:cases} (II), and (\ref{eqn:modulus}) implies that $\mu_{l_n}\to 0$. To see this, note that for any $\varepsilon>0$ we can choose $N$ such that $\beta_{l_n}\leq \omega(\varepsilon)$ for all $n\geq N$, and so $\mu_{l_n}\leq \varepsilon$ for all $n\geq N$. Therefore $\mu_n\to 0$ by Proposition \ref{prop:recineq2:conv}. For the other direction, assume now that $\mu_n\to 0$ and $\mu_n>0\implies \beta_n>0$. Given $\varepsilon>0$ let $N\in\NN$ be such that $\mu_n\leq \varepsilon$ for all $n\geq N$, and set
\begin{equation*}
\omega(\varepsilon):=\frac{1}{2}\cdot\min\{\beta_n>0\, | \, n< N\}
\end{equation*}
or just $\omega(\varepsilon):=1$ if $\beta_n=0$ for all $n<N$. Then supposing for contradiction that there is $n\in\NN$ such that $\beta_n\leq \omega(\varepsilon)$ and $\mu_n>\varepsilon$, then we must firstly have $\beta_n>0$ (using $\mu_n>0\implies \beta_n>0$) and secondly $n<N$ (using $\mu_n\leq \varepsilon$ for $n\geq N$). But then $\beta_n\leq \omega(\varepsilon)\leq \beta_n/2$, a contradiction for $\beta_n>0$.
\end{proof}
So far we have focused on convergence of $\seq{\mu_n}$. In most concrete applications, where $\beta_n=\psi(\mu_n)$, we also get $\beta_n\to 0$ as $n\to \infty$ from some basic property of $\psi$, e.g. continuity or monotonicity. We can also make this intuition formal in terms of a modulus going in the other direction:
\begin{proposition}
\label{prop:beta:conv}
Under the conditions of Proposition \ref{prop:recineq2}, together with the existence of a modulus $\omega$ satisfying (\ref{eqn:modulus}), suppose that there also exists some $\tilde\omega:\QQ_+\to \QQ_+$ such that for every $\varepsilon>0$ we have
\begin{equation}
\label{eqn:beta:modulus}
\mu_n\leq \tilde\omega(\varepsilon)\implies \beta_n\leq \varepsilon
\end{equation}
for all $n\in\NN$. Then $\beta_n\to 0$ as $n\to\infty$, and if $\phi$ is a rate of convergence for $\mu_n\to 0$ then $\phi\circ\tilde{\omega}$ is a rate of convergence for $\beta_n\to 0$. Conversely, if $\beta_n\to 0$ and $\beta_n>0\implies \mu_n>0$ for all $n\in\NN$, then there exists a $\tilde \omega$ satisfying (\ref{eqn:beta:modulus}) for all $\varepsilon>0$ and $n\in\NN$.
\end{proposition} 

\begin{proof}%checked
In one direction, we have $\mu_n\to 0$ by Proposition \ref{prop:recineq2}, and then $\beta_n\to 0$ follows easily from (\ref{eqn:beta:modulus}). The other direction is proved analogously to second part of the proof of Proposition \ref{prop:recineq2}.
\end{proof}

\begin{example}
\label{ex:monotone}%checked
Suppose that the conditions of Proposition \ref{prop:recineq2} hold and $\beta_n:=\psi(\mu_n)$ for some nondecreasing $\psi$ with $\psi(0)=0$ and $\psi(t)>0$ for $t>0$. Then (\ref{eqn:modulus}) holds for $\omega(\varepsilon):=\psi(\varepsilon)/2$, since if $\mu_n>\varepsilon$ then $\psi(\mu_n)\geq \psi(\varepsilon)>\omega(\varepsilon)$, and (\ref{eqn:beta:modulus}) holds for $\tilde\omega(\varepsilon):=\mu_{N}$ for the least $N$ such that $\beta_N\leq \varepsilon$ (this is well defined by the existence of a subsequence $\beta_{l_n}\to 0$). To see the latter, if $\mu_n\leq \mu_N$ then $\beta_n=\psi(\mu_n)\leq \psi(\mu_N)=\beta_N\leq \varepsilon$. Alternatively, if $\psi$ is also continuous at $0$, then $\tilde\omega$ can be defined to be any modulus of continuity.
\end{example}

The convergence theorems above play a role in numerous different convergence proofs in nonlinear analysis (some examples of which will be given in Section \ref{sec:applications}). Just as in the previous section, we now present some negative results which demonstrate the limitations on obtaining computable rates of convergence for $\mu_n\to 0$.
\begin{proposition}
\label{prop:recineq2:specker}

We can construct computable sequences $\seq{\mu_n}$, $\seq{\alpha_n}$, $\seq{\beta_n}$ and $\seq{\gamma_n}$ satisfying (\ref{eqn:recineq:basic}), with $\alpha_n>0$ for all $n\in\NN$, $\sum_{i=0}^\infty \alpha_i=\infty$, $\seq{\alpha_n}$ bounded and $\gamma_n/\alpha_n\to 0$, along with $\omega$ satisfying (\ref{eqn:modulus}), such that each of the following scenarios holds:

\begin{enumerate}[(a)]

	\item $\omega$ is computable, but $\mu_n\to 0$ with no computable rate of convergence;
	
	\item $\gamma_n/\alpha_n\to 0$ with computable rate of convergence, but $\mu_n\to 0$ with no computable rate of convergence.
	
\end{enumerate}
\end{proposition}

%\begin{proof}
%\begin{enumerate}[(a)]

%	\item Let $\seq{\mu_n}:=s_n$ for $\seq{s_n}$ defined as in Proposition %\ref{prop:specker:zero}, and  

%\end{enumerate}
%\end{proof}

\begin{proof}%checked
\begin{enumerate}[(a)]

	\item Define $\mu_n=\beta_n:=s_n$ as in Proposition \ref{prop:specker:zero} for any $\seq{a_n}$, $\alpha_n:=1$ and $\gamma_n:=s_{n+1}$. Then (\ref{eqn:recineq:basic}) is satisfied since
	\begin{equation*}
	\mu_n-\alpha_n\beta_n+\gamma_n=s_n-s_n+s_{n+1}=s_{n+1}=\mu_{n+1}
	\end{equation*}
	we have that (\ref{eqn:modulus}) is satisfied by $\omega(\varepsilon):=\varepsilon$, and $\gamma_n/\alpha_n=s_{n+1}\to 0$. But $\mu_n\to 0$ with no computable rate of convergence by construction.

	\item Define $\mu_n:=s_n$ as in Proposition \ref{prop:specker:zero} for $a_n:=1/(n+1)$, so that $s_n\leq 1$ for any $n\in\NN$. Define $\alpha_n:=n+1$, $\beta_n:=1/(n+1)^2$ and $\gamma_n:=2$. Then (\ref{eqn:recineq:basic}) is satisfied since
	\begin{equation*}
	\mu_n-\alpha_n\beta_n+\gamma_n=s_n-\frac{1}{n+1}+2\geq s_n+1\geq s_{n+1}=\mu_{n+1}
	\end{equation*}
	we have $\gamma_n/\alpha_n=2/(n+1)\to 0$ with computable rate of convergence and (\ref{eqn:modulus}) is satisfied by
	\begin{equation*}
		\omega(\varepsilon):=\frac{1}{(\phi(\varepsilon)+1)^2}
	\end{equation*}
	where $\phi$ is a rate of convergence for $\mu_n\to 0$, and $\phi$ cannot be computable since $\mu_n=s_n$ where $s_n$ is explicitly defined to have no computable rate of convergence.

\end{enumerate}
\end{proof}
This result shows that, in general, there is no way of obtaining a computable rate of convergence for $\mu_n\to 0$, even in cases where $\omega$ is computable, or we have a computable rate of convergence for $\gamma_n/\alpha_n\to 0$. We now provide a quantitative version of Proposition \ref{prop:recineq2:conv}, giving a rate of metastability for $\mu_n\to 0$ in the general case, which can be converted into a rate of convergence if \emph{both} $\omega$ is computable \emph{and} we have a computable rate of convergence for $\gamma_n/\alpha_n\to 0$. This is a metastable generalisation of Lemma 3.1 of the second author and Wiesnet \cite{powell-wiesnet:21:contractive}.
\begin{theorem}
\label{thm:recineq2:metastable}
Suppose that $\seq{\mu_n}$, $\seq{\alpha_n}$, $\seq{\beta_n}$ and $\seq{\gamma_n}$ are sequences of positive real numbers with $\alpha_n>0$ for all $n\in\NN$, satisfying (\ref{eqn:recineq:basic}), and that
\begin{enumerate}[(i)]

	\item $c>0$ is such that $\mu_n\leq c$ for all $n\in\NN$

	\item $\alpha>0$ is such that $\alpha_n\leq \alpha$ for all $n\in\NN$,

	\item $\sum_{i=0}^\infty \alpha_i=\infty$ with rate of divergence $r$,

\end{enumerate}
Fix $\varepsilon>0$ and $g:\NN\to\NN$, and define $h:\QQ_+\times \NN\to\NN$ by
\begin{equation*}
h(\delta,m):=\tilde g\left(r\left(m,\frac{2c}{\delta}\right)\right) \mbox{ \ \ \ for \ \ \ }\tilde g(i):=\max\{j+g(j)\, | \, j\leq i\}
\end{equation*}
Then if $p\in\NN$ and $\delta>0$ are such that there exists $m\leq p$ such that
\begin{equation}
\label{eqn:gamma:alpha}
\gamma_n/\alpha_n\leq \min\left\{\frac{\delta}{2},\frac{\varepsilon}{2\alpha}\right\} \ \ \ \mbox{for all} \ \ \ n\in [m, h(\delta,m)]
\end{equation}
and
\begin{equation}
\label{eqn:delta:mu}
\beta_n\leq\delta \implies \mu_n\leq \frac{\varepsilon}{2} \ \ \ \mbox{for all} \ \ \ n\leq h(\delta,p)
\end{equation}
then there exists some $N\leq r(p,2c/\delta)$ such that 
\begin{equation*}
\mu_n\leq \varepsilon \ \ \ \mbox{for all} \ \ \ n\in [N,N+g(N)]
\end{equation*}
\end{theorem}

\begin{proof}%checked
We first observe that for all $n\in [m,h(\delta,m)]$ it follows from (\ref{eqn:gamma:alpha}) that
\begin{equation}
\label{eqn:recineq2:meta0}
\mu_{n+1}\leq \mu_n-\alpha_n(\beta_n-\delta/2)
\end{equation}
and (using $\alpha_n\leq\alpha$) 
\begin{equation}
\label{eqn:recineq2:meta1}
\mu_{n+1}\leq \mu_n-\alpha_n\beta_n+\frac{\varepsilon}{2}
\end{equation}
Now suppose for contradiction that $\mu_n>\varepsilon/2$ for all $n\in [m,r(m,2c/\delta)]$. Then since
\begin{equation*}
n\leq r\left(m,\frac{2c}{\delta}\right)\leq h(\delta,m)\leq h(\delta,p)
\end{equation*}
where here we use the definition of $h$ and $\tilde{g}$, and the last step monotonicity of rates of divergence in their first argument, then from (\ref{eqn:delta:mu}) we also have $\beta_n>\delta$, and therefore using (\ref{eqn:recineq2:meta0}) we have:
\begin{equation*}
\mu_{n+1}\leq \mu_n-\alpha_n\cdot\frac{\delta}{2} \ \ \ \mbox{and thus} \ \ \ \alpha_n\cdot\frac{\delta}{2}<\mu_n-\mu_{n+1}
\end{equation*}
for $n$ in this range. Therefore
\begin{equation*}
\frac{\delta}{2}\sum_{i=m}^{r(m,2c/\delta)}\alpha_i<\mu_m-\mu_{r(m,2c/\delta)+1}\leq \mu_m\leq c
\end{equation*}
a contradiction. Therefore we have shown that there exists some $N\in [m,r(m,2c/\delta)]$ such that $\mu_N\leq\varepsilon/2$, and in particular, $N\leq r(m,2c/\delta)\leq r(p,2c/\delta)$.  We now show that in fact $\mu_n\leq \varepsilon$ for all $n\in [N,N+g(N)]$. We use induction, where the base case $n=N$ follows from the assumption. For the induction step we deal with two cases, first noting that $[N,N+g(N)]\subseteq [m,h(\delta,m)]$ by the definition of $h$ and $\tilde g$. If $\mu_n\leq \varepsilon/2$ then from (\ref{eqn:recineq2:meta1}) we have
\begin{equation*}
\mu_{n+1}\leq \mu_n-\alpha_n\beta_n+\varepsilon/2\leq \varepsilon
\end{equation*}
On the other hand, if $\varepsilon/2 <\mu_n\leq \varepsilon$ then we have $\beta_n>\delta$ by (\ref{eqn:delta:mu}) and $h(\delta,m)\leq h(\delta,p)$, and from (\ref{eqn:recineq2:meta0})
\begin{equation*}
\mu_{n+1}\leq \mu_n-\alpha_n(\beta_n-\delta/2)\leq \mu_n-\alpha_n\delta/2\leq \mu_n\leq \varepsilon
\end{equation*}
which completes the induction.
\end{proof}

\begin{corollary}
\label{cor:recineq2}
Suppose that $\seq{\mu_n}$, $\seq{\alpha_n}$, $\seq{\beta_n}$ and $\seq{\gamma_n}$ are sequences of positive real numbers with $\alpha_n>0$ for all $n\in\NN$, satisfying (\ref{eqn:recineq:basic}) and that
\begin{enumerate}[(i)]

	\item $c>0$ is such that $\mu_n\leq c$ for all $n\in\NN$,

	\item $\alpha>0$ is such that $\alpha_n\leq \alpha$ for all $n\in\NN$,

	\item $\sum_{i=0}^\infty \alpha_i=\infty$ with rate of divergence $r$,
	
	\item $\gamma_n/\alpha_n\to 0$ with rate of convergence $\phi$.

\end{enumerate}
Then whenever there exists a modulus $\omega:(0,\infty)\to (0,\infty)$ such that for all $\varepsilon>0$ and $n\in\NN$:
\begin{equation*}
\beta_n\leq \omega(\varepsilon)\implies \mu_n\leq \varepsilon
\end{equation*}
then $\mu_n\to 0$ with rate of convergence
\begin{equation*}
\Phi(\varepsilon):=r\left(\phi\left(\min\left\{\frac{\omega(\varepsilon/2)}{2},\frac{\varepsilon}{2\alpha}\right\}\right),\frac{2c}{\omega(\varepsilon/2)}\right)
\end{equation*}
\end{corollary}

\begin{proof}%checked
Directly from Theorem \ref{thm:recineq2:metastable}, defining $\delta:=\omega(\varepsilon/2)$ and
\begin{equation*}
p:=\phi\left(\min\left\{\frac{\delta}{2},\frac{\varepsilon}{2\alpha}\right\}\right)
\end{equation*}
and noting that these satisfy the premise of the theorem for any $g:\NN\to\NN$.
\end{proof}
%

%%%%%%%%%%%%%%%%%%%%%%%%%%%%%%%%%%%%%%%%%%%%%%%%%%%%%%%%%%%%%%%%%%%%%%%%%%%%%%%%%%%%%%%%%%
\subsubsection{Second variant of Case II}
\label{sec:partii:second}
%%%%%%%%%%%%%%%%%%%%%%%%%%%%%%%%%%%%%%%%%%%%%%%%%%%%%%%%%%%%%%%%%%%%%%%%%%%%%%%%%%%%%%%%%%

We conclude the section by considering a variant of Proposition \ref{prop:recineq2:conv} whose only difference is a slight change in the modulus condition (replacing $\mu_n$ with $\mu_{n+1}$), but which comes with a slightly simpler computational interpretation, and, as we will see in Section \ref{sec:applications}, has already led to interesting applications in a number of areas. The following is a metastable variant of Lemma 3.4 of Kohlenbach and the second author \cite{kohlenbach-powell:20:accretive}, which generalises that result:

\begin{theorem}
\label{thm:recineq2:metastable:v2}
Suppose that $\seq{\mu_n}$, $\seq{\alpha_n}$, $\seq{\beta_n}$ and $\seq{\gamma_n}$ are sequences of nonnegative real numbers with $\alpha_n>0$ for $n\in\NN$, satisfying (\ref{eqn:recineq:basic}), and that
\begin{enumerate}[(i)]

	\item $c>0$ is such that $\mu_n\leq c$ for all $n\in\NN$

	\item $\alpha>0$ is such that $\alpha_n\leq \alpha$ for all $n\in\NN$,

	\item $\sum_{i=0}^\infty \alpha_i=\infty$ with rate of divergence $r$,

\end{enumerate}
Fix $\varepsilon>0$ and $g:\NN\to\NN$, and define $h:\QQ_+\times\NN\to\NN$ by
\begin{equation*}
h(\delta,m):=\tilde g\left(r\left(m,\frac{2c}{\delta}\right)\right) \ \ \ \mbox{with} \ \ \ \tilde g(i):=\max\{j+g(j)\, | \, j\leq i\} 
\end{equation*}
Then if $p\in\NN$ and $\delta>0$ are such that there exists $m\leq p$ such that
\begin{equation*}
\gamma_n/\alpha_n\leq \frac{\delta}{2} \ \ \ \mbox{for all} \ \ \ n\in [m, h(\delta,m)]
\end{equation*}
and
\begin{equation*}
\beta_{n}\leq\delta \implies \mu_{n+1}\leq \varepsilon \ \ \ \mbox{for all} \ \ \ n\leq h(\delta,p)
\end{equation*}
then there exists some $N\leq r(m,2c/\delta)+1$ such that 
\begin{equation*}
\mu_n\leq \varepsilon \ \ \ \mbox{for all} \ \ \ n\in [N,N+g(N)]
\end{equation*}
\end{theorem}

\begin{proof}%checked
Suppose for contradiction that $\mu_n>\varepsilon$ for all $n\in [m,r(m,2c/\delta)+1]\subseteq [m,h(\delta,m)]$. Then for $n\leq r(m,2c/\delta)$ we also have $n\leq h(\delta,p)$, and since $\mu_{n+1}>\varepsilon$ we have $\beta_n>\delta$ for $n\in [m,r(m,2c/\delta)]$. It follows that
\begin{equation*}
\mu_{n+1}\leq \mu_n-\alpha_n\delta+\alpha_n\frac{\delta}{2}=\mu_n-\alpha_n\frac{\delta}{2}
\end{equation*}
and so
\begin{equation*}
\frac{\delta}{2}\sum_{i=m}^{r(m,2c/\delta)}\alpha_i<\mu_m-\mu_{r(m/2c\delta)+1}\leq c
\end{equation*}
a contradiction. So we have shown that there exists some $N\in [m,r(m,2c/\delta)+1]$ such that $\mu_N\leq \varepsilon$. We now show that $\mu_n\leq \varepsilon$ for all $n\in [N,N+g(N)]$. Suppose for contradiction that there is some $n\in [N,N+g(N)]$ such that $\mu_n\leq \varepsilon$ but $\mu_{n+1}>\varepsilon$. Then since $n\leq h(\delta,p)$ we have $\beta_n>\delta$ and thus
\begin{equation*}
\varepsilon<\mu_{n+1}\leq \mu_n-\alpha_n\frac{\delta}{2}\leq\varepsilon
\end{equation*}
a contradiction.
\end{proof}

\begin{corollary}
\label{cor:recineq2b}
Suppose that $\seq{\mu_n}$, $\seq{\alpha_n}$, $\seq{\beta_n}$ and $\seq{\gamma_n}$ are sequences of nonnegative real numbers with $\alpha_n>0$ for $n\in\NN$, satisfying (\ref{eqn:recineq:basic}), and that
\begin{enumerate}[(i)]

	\item $c>0$ is such that $\mu_n\leq c$ for all $n\in\NN$

	\item $\alpha>0$ is such that $\alpha_n\leq \alpha$ for all $n\in\NN$,

	\item $\sum_{i=0}^\infty \alpha_i=\infty$ with rate of divergence $r$,
	
	\item $\gamma_n/\alpha_n\to 0$ with rate of convergence $\phi$.

\end{enumerate}
Then whenever there exists a modulus $\omega:(0,\infty)\to (0,\infty)$ such that for all $\varepsilon>0$ and $n\in\NN$:
\begin{equation*}
\beta_{n}\leq \omega(\varepsilon)\implies \mu_{n+1}\leq \varepsilon
\end{equation*}
then $\mu_n\to 0$ with rate of convergence
\begin{equation*}
\Phi(\varepsilon):=r\left(\phi\left(\frac{\omega(\varepsilon)}{2}\right),\frac{2c}{\omega(\varepsilon)}\right)+1
\end{equation*}
where $c>0$ is any upper bound on $\seq{\mu_n}$.
\end{corollary}

\begin{proof}%checked 
Directly from Theorem \ref{thm:recineq2:metastable:v2}, defining $\delta:=\omega(\varepsilon)$ and $p:=\phi(\delta/2)$, and observing that these satisfy the premise of the theorem for any $g:\NN\to\NN$.
\end{proof}

%%%%%%%%%%%%%%%%%%%%%%%%%%%%%%%%%%%%%%%%%%%%%%%%%%%%%%%%%%%%%%%%%%%%%%%%%%%%%%%%%%%%%%%%%%
%%%%%%%%%%%%%%%%%%%%%%%%%%%%%%%%%%%%%%%%%%%%%%%%%%%%%%%%%%%%%%%%%%%%%%%%%%%%%%%%%%%%%%%%%%
\section{Applications}
\label{sec:applications}
%%%%%%%%%%%%%%%%%%%%%%%%%%%%%%%%%%%%%%%%%%%%%%%%%%%%%%%%%%%%%%%%%%%%%%%%%%%%%%%%%%%%%%%%%%
%%%%%%%%%%%%%%%%%%%%%%%%%%%%%%%%%%%%%%%%%%%%%%%%%%%%%%%%%%%%%%%%%%%%%%%%%%%%%%%%%%%%%%%%%%

In this second main section we demonstrate how our quantitative study of recursive inequalities can be applied to obtain numerical results from theorems in nonlinear analysis. Our main contribution is a new case study involving the subclass explored in Section \ref{sec:recineq:parti}, where we apply our rate of metastability for $\beta_n\to 0$ to obtain quantitative information for a class of gradient descent algorithms. Instances of the results from Section \ref{sec:recineq:partii} have already been used in a number of recent papers, and we also provide a high-level outline some of these. In all cases, we start by presenting a simplified form of the application, motivating how the recursive inequality and associated conditions emerge naturally from the concrete setting, before discussing more detailed results.

%%%%%%%%%%%%%%%%%%%%%%%%%%%%%%%%%%%%%%%%%%%%%%%%%%%%%%%%%%%%%%%%%%%%%%%%%%%%%%%%%%%%%%%%%%
\subsection{Recursive inequalities in applied proof theory}
\label{sec:applications:lit}
%%%%%%%%%%%%%%%%%%%%%%%%%%%%%%%%%%%%%%%%%%%%%%%%%%%%%%%%%%%%%%%%%%%%%%%%%%%%%%%%%%%%%%%%%%

Before we start, we give a very brief survey of some of the different types of recursive inequality that have been recently studied from a quantitative perspective in applied proof theory. We make no attempt to give a detailed or comprehensive overview of this topic (which would constitute a paper in its own right) - rather we content ourselves with a few pointers to representative papers in the proof theory literature, and explain how the inequalities used there relate to the class we study here. 

Some of the first quantitative recursive inequalities appearing in the proof theory literature involve the general inequality already discussed in Section \ref{sec:recineq:parti}:
\begin{equation*}
\mu_{n+1}\leq (1+\delta_n)\mu_n+\gamma_n
\end{equation*}
which we used to show that $\seq{\mu_n}$ converges to some limit under the conditions $\sum_{i=0}^\infty \gamma_i<\infty$ and $\sum_{i=0}^\infty \delta_i<\infty$. Rates of metastability for $\seq{\mu_n}$ have been extracted, and then applied to obtain, for instance, bounds on the computation of approximate fixed points of asymptotically quasi-nonexpansive mappings by Kohlenbach and Lambov \cite{kohlenbach-lambov:04:asymptotically:nonexpansive}, or for nonexpansive mappings in uniformly convex hyperbolic spaces by Kohlenbach and Leu\c{s}tean \cite{kohlenbach-leustean:10:asymp:nonexpansive:hyperbolic}.

If we consider the special case of (\ref{eqn:recineq:basic}) where $\beta_n=\mu_n$ i.e.
\begin{equation}
\label{eqn:special}
\mu_{n+1}\leq (1-\alpha_n)\mu_n+\gamma_n
\end{equation}
we can show that $\mu_n\to 0$ for both $\sum_{i=0}^\infty\gamma_i<\infty$ and $\gamma_n/\alpha_n\to 0$.  In fact, in this case the proofs can be simplified and the two conditions combined, and quantitative results giving both direct and metastable rates of convergence have been crucial in numerous different contexts. An early use of kind of recursive inequality was to extract rates of asymptotic regularity for the Halpern iterations of nonexpansive self-mappings by Leu\c{s}tean \cite{leaustean:07:halpern}, and a particularly detailed discussion of variants of (\ref{eqn:special}) is given by Kohlenbach and Leu\c{s}tean \cite{kohlenbach-leustean:12:metastability}. There have been many further examples since: The last years have seen numerous instances of (\ref{eqn:special}) with combined conditions by several different authors, including \cite{cheval-leustean:21:tikhonovmann,dinis-pinto:20:proximal,dinis-pinto:21:strong,dinis-pinto:21:effective,dinis-pinto:21:proximal,kohlenbach:20:halpern,kohlenbach-pinto:21:viscosity,leaustean-nicolae:14:compositions,leaustean-nicolae:16:kappa,leaustean-pinto:21:halpern,pinto:21:halpern,sipos:22:abstract}.

Less common but perhaps more relevant to us are examples of the subclass considered in Section \ref{sec:recineq:partii} for $\beta_n=\psi(\mu_n)$ or $\beta_n=\psi(\mu_{n+1})$. Here rates of convergence closely related to Corollary \ref{cor:recineq2} or \ref{cor:recineq2b} have been applied in various different contexts. The main variant $\beta_n=\psi(\mu_n)$ was crucial in a recent paper by the second author and Wiesnet on rates of convergence for generalized asymptotically weakly contractive mappings \cite{powell-wiesnet:21:contractive}, while the second variant $\beta_n=\psi(\mu_{n+1})$ considered in Section \ref{sec:partii:second} was initially used by Kohlenbach and K\"ornlein \cite{kohlenbach-koernlein:11:pseudocontractive} to extract rates of convergence for pseudocontractive mappings. More recently, it has featured in the work of Kohlenbach and the second author \cite{kohlenbach-powell:20:accretive}, where Corollary \ref{cor:recineq2b} is applied to obtain rates of convergence for algorithms involving set-valued accretive operators, and Sipo\c{s} \cite{sipos:21:jointly:nonexpansive}, where the same result is used in the context of jointly nonexpansive mappings. Both \cite{kohlenbach-powell:20:accretive} and \cite{powell-wiesnet:21:contractive} will be discussed in more detail in Sections \ref{sec:applications:weakly:contractive} and \ref{sec:applications:acc} below.

%%%%%%%%%%%%%%%%%%%%%%%%%%%%%%%%%%%%%%%%%%%%%%%%%%%%%%%%%%%%%%%%%%%%%%%%%%%%%%%%%%%%%%%%%%
\subsection{Application of Case I: A quantitative analysis of a class of gradient descent methods}
\label{sec:applications:gradient}
%%%%%%%%%%%%%%%%%%%%%%%%%%%%%%%%%%%%%%%%%%%%%%%%%%%%%%%%%%%%%%%%%%%%%%%%%%%%%%%%%%%%%%%%%%

We now give a new quantitative result that utilises Proposition \ref{prop:recineq1:beta:generalise} (in the form of Corollary \ref{cor:recineq1:beta:rates}) in the context of gradient descent algorithms. As such, the content of this section should be seen as a new contribution to applied proof theory in nonlinear analysis, but we aim to present our work without assuming any technical background in the area of application. Gradient descent algorithms have been the subject of case studies in applied proof theory before, notably rates of metastability for a hybrid steepest descent method was given by K\"ornlein \cite{koernlein:pp:gradient} (which also appears as Chapter 8 of K\"ornlein's thesis \cite{koernlein:16:phd}), who analyses a theorem of Yamada \cite{yamada:01:hybrid}, and a quantitative analysis of a subgradient-like algorithm was recently given by Kohlenbach and Pischke \cite{pischke-kohlenbach:21:subgradient}, who extract computational information from a theorem of Iiduke and Yamada \cite{iiduke-yamada:09:subgradient}. Our contribution here, and in particular the concrete application on projective methods in Section \ref{sec:applications:gradient:projective}, is separate to these (K\"ornlein considers a different algorithm, while Kohlenbach and Pischke work in $\RR^n$), but it would nevertheless be interesting to explore any parallels between these approaches in future work.

To help motivate what follows, we first consider a simple example:
\begin{example}[Gradient descent in $\RR^n$]
\label{ex:gradient:Rn}
Let $f:\RR^n\to \RR$ be a convex, differentiable function, and consider the usual steepest descent algorithm defined by
\begin{equation}
\label{eqn:ip:seq}
x_{n+1}=x_n-\alpha_n\nabla f(x_n)
\end{equation}
for some initial $x_0$, where the step sizes $\seq{\alpha_n}$ satisfy
\begin{equation*}
\sum_{i=0}^\infty \alpha_i=\infty \ \ \ \mbox{and} \ \ \ \sum_{i=0}^\infty \alpha_i^2<\infty
\end{equation*}
Assuming that $L>0$ is such that $\norm{\nabla f(x_n)}\leq L$ for all $n\in\NN$, by (\ref{eqn:ip:seq}) we have
\begin{equation}
\label{eqn:ip:reg}
\norm{x_{n+1}-x_n}\leq L\alpha_n
\end{equation}
Now, let $x^\ast$ be a point where $f$ attains its minimum i.e. $f(x^\ast)\leq f(y)$ for all $y\in\RR^n$. A basic property of a differentiable function $f:\RR^n\to \RR$ being convex is that
\begin{equation}
\label{eqn:ip:con}
f(x)\geq f(y)+\ip{\nabla f(y),x-y}
\end{equation}
for all $x,y\in \RR^n$, and therefore in particular
\begin{equation}
\label{eqn:ip:xn}
0\leq f(x_n)-f(x^\ast)\leq \ip{\nabla f(x_n),x_n-x^\ast}
\end{equation}
for all $n\in\NN$. Setting $\beta_n:=f(x_n)-f(x^\ast)$ we have
\begin{equation}
\label{eqn:ip:calc}
\begin{aligned}
&2\alpha_n\beta_n\leq 2\ip{\alpha_n\nabla f(x_n),x_n-x^\ast} \ \ \ \mbox{ by (\ref{eqn:ip:xn})}\\
&=2\ip{x_n-x_{n+1},x_n-x^\ast} \ \ \ \mbox{by (\ref{eqn:ip:seq})}\\
&=\norm{x_{n+1}-x_n}^2+\norm{x_n-x^\ast}^2-\norm{x_{n+1}-x^\ast}^2 \ \ \ \mbox{(properties of inner product)}\\
&\leq L^2\alpha_n^2+\norm{x_n-x^\ast}^2-\norm{x_{n+1}-x^\ast}^2 \ \ \ \mbox{ by (\ref{eqn:ip:reg})}
\end{aligned}
\end{equation} 
and therefore the main recursive inequality (\ref{eqn:recineq:basic}) is satisfied with 
\begin{equation*}
\mu_n:=\frac{1}{2}\norm{x_n-x^\ast}^2\mbox{ \ \ \ and \ \ \ }\gamma_n:=\frac{1}{2}L^2\alpha_n^2
\end{equation*}
Moreover, from $\sum_{i=0}^\infty \alpha_i^2<\infty$ we clearly have $\sum_{i=0}^\infty \gamma_i<\infty$ and, using convexity of $f$ together with the Cauchy-Schwartz inequality and (\ref{eqn:ip:reg}):
\begin{equation*}
\begin{aligned}
\beta_n-\beta_{n+1}&=f(x_n)-f(x_{n+1})\leq \ip{\nabla f(x_n),x_n-x_{n+1}}\\
&\leq \norm{\nabla f(x_n)}\cdot \norm{x_n-x_{n+1}}\leq L^2\alpha_n
\end{aligned}
\end{equation*}
and thus the assumptions of Proposition \ref{prop:recineq1:beta:conv} are satisfied for $\theta:=L^2$. Therefore by the conclusion of that proposition, it follows that $f(x_n)\to f(x^\ast)$, and moreover, using the quantitative results of Section \ref{sec:recineq:parti} we can produce a computable rate of metastability in terms of suitable input data.
\end{example}

We stress that for the special case of gradient descent in $\RR^n$ sketched above, more direct proofs of convergence are possible, with better rates of convergence. However, there exist a number of much more general results whose proofs use Proposition \ref{prop:recineq1:beta:conv} or its variants in a crucial way, and in these cases, our quantitative results give us a means to extract computational information from those proofs. 

A concrete example of such a result will be considered in Section \ref{sec:applications:gradient:projective} below, but we first give a new, abstract convergence theorem that axiomatises some key properties of steepest descent used above, making precise what we need in order to be able to reduce convergence to Proposition \ref{prop:recineq1:beta:conv}. Note that this result applies to an arbitrary real inner product space, and does not explicitly introduce a gradient descent method. Instead, this is implicit in the abstract properties that we require.

We start by motivating these abstract properties. Let $f:X\to \RR$ be a function for some real inner product space $X$, $\seq{x_n}$ and $\seq{u_n}$ be sequences of vectors in $X$, and $\seq{\alpha_n}$ a sequence of nonnegative reals, with the intuitive idea being that $\seq{x_n}$ is generated via some kind of gradient descent algorithm, where $\seq{u_n}$ denotes the gradients of $f$ at points $\seq{x_n}$, and $\seq{\alpha_n}$ the step sizes i.e.
\begin{equation}
\label{eqn:abstract:alg}
x_{n+1}=x_n-\alpha_nu_n
\end{equation}
Looking in more detail at our simple example, the precise definition of the algorithm is used to establish two key properties, namely
\begin{enumerate}[(I)]

	\item $\norm{x_{n+1}-x_n}\leq L\alpha_n$ (cf. (\ref{eqn:ip:reg}))
	
	\item $\ip{\alpha_nu_n,x_n-x^\ast}=\ip{x_n-x_{n+1},x_n-x^\ast}$ (cf. (\ref{eqn:ip:calc}))

\end{enumerate}
As we will see, we can replace (\ref{eqn:abstract:alg}) with general versions of these two properties, namely
\begin{enumerate}[(I)]

	\item $\norm{x_{n+1}-x_n}\leq c_n$

	\item $\ip{\alpha_nu_n,x_n-x^\ast}\leq a\ip{x_n-x_{n+1},x_n-x^\ast}+d_n$

\end{enumerate}
for some constant $a>0$ and sequences of nonnegative reals $\seq{c_n}$, $\seq{d_n}$ with $\sum_{i=0}^\infty c^2_i<\infty$ and $\sum_{i=0}^\infty d_i<\infty$, and still reduce the problem to an instance of (\ref{eqn:recineq:basic}). In a similar way, the fact that $u_n$ represents a gradient of $f$ at $x_n$ is only used to establish
\begin{equation*}
f(x_n)-f(y)\leq \ip{u_n,x_n-y}
\end{equation*} 
and again, we will show that this can be generalised to incorporate an error term i.e.
\begin{equation*}
f(x_n)-f(y)\leq \ip{u_n,x_n-y}+b_n
\end{equation*}
for some sequence of nonnegative reals $\seq{b_n}$ with $\sum_{i=0}^\infty \alpha_ib_i<\infty$. We now use these properties to present an abstract, quantitative convergence theorem for a class of gradient descent-like methods, where we use Proposition \ref{prop:recineq1:beta:conv} to establish convergence, and Corollary \ref{cor:recineq1:beta:rates} to provide a rate of metastability. We certainly do not claim that this abstract version provides any kind of definitive characterisation of gradient descent methods in inner product spaces: It simply gives an abstract presentation of a type of convergence proof that utilises a specific reduction to (\ref{eqn:recineq:basic}). We give a nontrivial example of how this can then be applied in a concrete setting in Section \ref{sec:applications:gradient:projective}.
\begin{theorem}
\label{thm:abstract:gradient}
Let $X$ be a real inner product space with $Y\subseteq X$, and suppose that $f:X\to \RR$ is a function. Let $\seq{\alpha_n}$ be a sequence of nonnegative reals with $\sum_{i=0}^\infty\alpha_i=\infty$, and $\seq{b_n}, \seq{c_n}$ and $\seq{d_n}$ sequences of nonnegative reals with $\sum_{i=0}^\infty \alpha_ib_i<\infty$, $\sum_{i=0}^\infty c^2_i<\infty$ and $\sum_{i=0}^\infty d_i<\infty$. Finally, suppose that $x^\ast\in Y$, $\seq{x_n}$ and $\seq{u_n}$ are sequences of vectors with $x_n\in Y$ for all $n\in\NN$, and $a,p,\theta>0$ are constants, which satisfy the following properties for all $n\in\NN$:
\begin{enumerate}[(i)]

	\item\label{cond:min} $f(x^\ast)\leq f(x_n)$ ($x^\ast$ acts as a minimizer)
		
	\item\label{cond:grad} $f(x_n)-f(y)\leq \ip{u_n,x_n-y} + b_n$ for all $y\in Y$ ($u_n$ acts as a gradient)

	\item\label{cond:descent:i} $\norm{x_{n+1}-x_n}\leq c_n$ ($x_n$ acts as a gradient descent method, property I)
	
	\item\label{cond:descent:ii} $\ip{\alpha_nu_n,x_n-x^\ast}\leq a\ip{x_n-x_{n+1},x_n-x^\ast}+d_n$ ($x_n$ acts as a gradient descent method, property II)
	
	\item\label{cond:bound} $\norm{u_n}\leq p$ (gradients are bounded)
	
	\item\label{cond:coeff} $pc_n+b_n\leq \theta\alpha_n$

\end{enumerate}
Then $f(x_n)\to f(x^\ast)$. Moreover, if $r$ is a rate of divergence for $\sum_{i=0}^\infty\alpha_i=\infty$ and $b,c,d,K>0$ are such that 
\begin{equation*}
\sum_{i=0}^\infty \alpha_ib_i\leq b, \ \ \ \sum_{i=0}^\infty c^2_i\leq c, \ \ \ \sum_{i=0}^\infty d_i\leq d, \mbox{ \ \ \ and \ \ \ } \norm{x_0-x^\ast}^2\leq K
\end{equation*}
Then for all $\varepsilon>0$ and $g:\NN\to\NN$ we have
\begin{equation*}
\exists n\leq \Phi(\varepsilon,g)\, \forall k\in [n,n+g(n)]\, (f(x_k)\leq f(x^\ast)+\varepsilon)
\end{equation*} 
where
\begin{equation*}
\begin{aligned}
\Phi(\varepsilon,g)&:=\tilde{h}^{(\lceil 4\theta e/\varepsilon^2\rceil)}(0)\\
\tilde{h}(n)&:=r\left(n+g(n),\frac{\varepsilon}{2\theta}\right)+1\\[2mm]
e&:=\frac{a(c+K)}{2}+b+d
\end{aligned}
\end{equation*}
\end{theorem}
%
%
%We discuss each of the conditions (i)--(v) in turn. The first simply says that $f(x^\ast)$ is a lower bound for $\seq{f(x_n)}$, which in particular would be satisfied by any minimizer $x^\ast$ of $f$ on $C$. The second characterises $u_n$ as a gradient of $f$ at $x_n$ and $f$ being convex, generalising (\ref{eqn:ip:xn}) and (\ref{eqn:ip:con}). The Third is just (\ref{eqn:ip:reg}), now taken directly as an assumption. The fourth generalises the first line of (\ref{eqn:ip:calc}), characterising $\seq{x_n}$ as a steepest-descent style algorithm, while. Finally, (v) states that the gradient has bounded size at points $\seq{x_n}$. Each of the conditions is easily seen to be satisfied by steepest descent in $\RR^n$ as already sketched. However, Section \ref{sec:applications:gradient:projective} below we show that they also hold in a much more general setting. 

To prove Theorem \ref{thm:abstract:gradient}, we require the following lemma, which is a simple variant of a standard rate of metastability for monotone increasing sequences.
\begin{lemma}
\label{lem:metastability}
Let $\seq{b_n}$ be a sequence of nonnegative numbers with $\sum_{i=0}^\infty b_i\leq b$. Then a rate of metastability for $\sum_{i=0}^\infty b_i<\infty$ is given by
\begin{equation*}
\Phi(\varepsilon,g):=\tilde g^{(\lceil b/\varepsilon\rceil)}(0) \ \ \ \mbox{where} \ \ \ \tilde g(n):=n+g(n)+1
\end{equation*}
\end{lemma}

\begin{proof}
Assume that the statement is false i.e. there exist $\varepsilon>0$ and $g:\NN\to\NN$ such that
\begin{equation*}
\sum_{j=n}^{n+g(n)}b_i>\varepsilon \ \ \ \mbox{for all $n\leq \Phi(\varepsilon,g)$}
\end{equation*}
Then for any $i<\lceil b/\varepsilon\rceil$ we have $\tilde g^{(i)}(0)\leq \Phi(\varepsilon,g)$ and thus, using that $\tilde g^{(i)}(0)+g(\tilde g^{(i)}(0))=\tilde g^{(i+1)}(0)-1$, we have
\begin{equation*}
\sum_{j=\tilde g^{(i)}(0)}^{\tilde g^{(i+1)}(0)-1}b_i>\varepsilon
\end{equation*}
and so summing over all $i<\lceil b/\varepsilon\rceil$:
\begin{equation*}
\sum_{i=0}^{\Phi(\varepsilon,g)-1}b_i=\sum_{i=0}^{\lceil b/\varepsilon\rceil-1}\left(\sum_{j=\tilde g^{(i)}(0)}^{\tilde g^{(i+1)}(0)-1}b_j\right)>\sum_{i=0}^{\lceil b/\varepsilon\rceil-1}\varepsilon=\lceil b/\varepsilon\rceil\cdot \varepsilon\geq b
\end{equation*}
contradicting $\sum_{i=0}^\infty b_i\leq b$.
\end{proof}

\begin{proof}[Proof of Theorem \ref{thm:abstract:gradient}]
Setting $\beta_n:=f(x_n)-f(x^\ast)$ and noting that $\beta_n\geq 0$ by (\ref{cond:min}), we have:
\begin{equation}
\begin{aligned}
2\alpha_n\beta_n&\leq 2\ip{\alpha_nu_n,x_n-x^\ast} +2\alpha_nb_n \ \ \ \mbox{by (\ref{cond:grad})}\\
&\leq 2a\ip{x_n-x_{n+1},x_n-x^\ast}+2(\alpha_nb_n+d_n) \ \ \ \mbox{by (\ref{cond:descent:ii})}\\
&=a(\norm{x_{n+1}-x_n}^2+\norm{x_n-x^\ast}^2-\norm{x_{n+1}-x^\ast}^2)+2(\alpha_nb_n+d_n)\\
&\leq a(\norm{x_n-x^\ast}^2-\norm{x_{n+1}-x^\ast}^2)+[ac^2_n+2(\alpha_nb_n+d_n)] \ \ \ \mbox{by (\ref{cond:descent:i})}
\end{aligned} 
\end{equation}
where the third line is by a routine calculation and standard properties of real inner products, and therefore (\ref{eqn:recineq:basic}) is satisfied for 
\begin{equation*}
\mu_n:=\frac{a}{2}\norm{x_n-x^\ast}^2\mbox{ \ \ \ and \ \ \ }\gamma_n:=\frac{ac^2_n
}{2}+\alpha_nb_n+d_n
\end{equation*}
In particular, it follows that
\begin{equation}
\label{eqn:defb}
\begin{aligned}
\sum_{i=0}^\infty \alpha_i\beta_i&\leq \mu_0+\sum_{i=0}^\infty\gamma_i\\
&=\frac{a}{2}\norm{x_0-x^\ast}^2+\sum_{i=0}^\infty \left(\frac{ac^2_n
}{2}+\alpha_nb_n+d_n\right)\\
&\leq \frac{aK}{2}+\frac{ac}{2}+b+d=:e
\end{aligned}
\end{equation}
We also have
\begin{equation*}
\begin{aligned}
\beta_n-\beta_{n+1}&=f(x_n)-f(x_{n+1})\\
&\leq \ip{u_n,x_n-x_{n+1}} + b_n \ \ \ \mbox{by (\ref{cond:grad})}\\
&\leq \norm{u_n}\norm{x_n-x_{n+1}} + b_n \ \ \ \mbox{Cauchy-Schwartz}\\
&\leq pc_n+b_n \ \ \ \mbox{by (\ref{cond:grad}) and (\ref{cond:bound})}\\
&\leq \theta\alpha_n \ \ \ \mbox{by (\ref{cond:coeff})}
\end{aligned}
\end{equation*}
therefore the conditions of Proposition \ref{prop:recineq1:beta:conv} are satisfied for $\seq{\beta_n}$ defined as above, and thus $f(x_n)-f(x^\ast)\to 0$. To get the rate of metastability, we apply Corollary \ref{cor:recineq1:beta:rates} (b). Since $\sum_{i=0}^\infty\alpha_i\beta_i\leq e$, by Lemma \ref{lem:metastability}, rate of metastability for $\sum_{i=0}^\infty \alpha_i\beta_i<\infty$ is therefore given by 
\begin{equation*}
\Psi(\varepsilon,h):=\tilde h^{\lceil e/\varepsilon\rceil}(0) \ \ \ \mbox{where} \ \ \ \tilde{h}(n):= n+h(n)+1
\end{equation*}
for $e$ as defined in (\ref{eqn:defb}) above, and therefore by Corollary \ref{cor:recineq1:beta:rates} a rate of metastability for $f(x_n)-f(x^\ast)\to 0$ is given by
\begin{equation*}
\Phi(\varepsilon,g):=\Psi\left(\frac{\varepsilon^2}{4\theta},h\right)
\end{equation*}
where $h(n)=r(n+g(n),\varepsilon/2\theta)-n$. Unwinding the definitions gives the result.
\end{proof}

\begin{remark}
A rate of metastability for $\seq{f(x_n)}$ in the special case of gradient descent in $\RR^n$ as sketched in Example \ref{ex:gradient:Rn} can be read off from Theorem \ref{thm:abstract:gradient} by setting $e:=(L+K)/2$ and $\theta:=L^2$, though as emphasised already, alternative proofs for this simple case will yield better quantitative information. 
\end{remark}

%%%%%%%%%%%%%%%%%%%%%%%%%%%%%%%%%%%%%%%%%%%%%%%%%%%%%%%%%%%%%%%%%%%%%%%%%%%%%%%%%%%%%%%%%%
\subsubsection{An application for projective subgradient methods}
\label{sec:applications:gradient:projective}
%%%%%%%%%%%%%%%%%%%%%%%%%%%%%%%%%%%%%%%%%%%%%%%%%%%%%%%%%%%%%%%%%%%%%%%%%%%%%%%%%%%%%%%%%%

We now use our abstract theorem to give a quantitative version of a result by Alber, Iusem and Solodov \cite{alber-iusem-solodov:98:subgradient}, where convergence $f(x_n)\to f(x^\ast)$ for a generalised subgradient method $\seq{x_n}$ towards a  minimizer $x^\ast$ of $f$ is established, and then used to prove weak convergence of $\seq{x_n}$. We provide a rate of metastability for $\seq{f(x_n)}$, and conjecture that similar results on projective subgradient methods e.g. \cite{alber-iusem:01:subgradient,alber-iusem-solodov:98:subgradient,drummond-iusem:04:projected:gradient,mainge:08:projected:subgradient} could also be reduced to Theorem \ref{thm:abstract:gradient} (or a simple variant thereof), although we leave the details to future work. 

Suppose that $H$ is a real valued Hilbert space, $C\subseteq H$ a closed, convex subset of $H$, and $f:H\to \RR$ a convex and continuous function. The $\varepsilon$-subdifferential of $f$ at $x\in H$ is defined by
\begin{equation*}
\partial_\varepsilon f(x):=\left\{u\in H\, | \, f(y)-f(x)\geq \ip{u,y-x}-\varepsilon \ \ \ \mbox{for all $y\in H$}\right\}
\end{equation*} 
with the case $\varepsilon=0$ coinciding with the usual subdifferential $\partial f(x)$. Letting $P_C:H\to C$ be the orthogonal projection of $H$ into $C$, Alber. et al consider the following algorithm:
\begin{equation}
\label{eqn:subgradient}
x_{n+1}=P_C\left(x_n-\frac{\alpha_n}{\nu_n}u_n\right) \ \ \ \mbox{for} \ \ \ u_n\in\partial_{\varepsilon_n} f(x_n) \ \ \ \mbox{with} \ \ \  u_n\neq 0
\end{equation}
where $\seq{\alpha_n}$ satisfies $\sum_{i=0}^\infty\alpha_i=\infty$ and $\sum_{i=0}^\infty \alpha^2_i<\infty$, $\seq{\varepsilon_n}$ is a sequence of nonnegative error terms with $\varepsilon_n\leq \mu\alpha_n$ for some $\mu>0$ and $\nu_n:=\max\{1,\norm{u_n}\}$. The algorithm halts if $0\in\partial_{\varepsilon_n}f(x_n)$ at any point.
\begin{theorem}
[Cf. Lemma 1 and Theorems 1 and 2 of Alber, Iusem and Solodov \cite{alber-iusem-solodov:98:subgradient}]
\label{thm:projective}
Let $x^\ast\in C$ be a minimizer of $f$ on $C$, and suppose that $\seq{x_n}$ is an infinite sequence generated by the algorithm (\ref{eqn:subgradient}), whose components satisfy all of the properties outlined above. Suppose that $\rho>1$ is such that $\norm{u_n}\leq \rho$ for all $n\in\NN$. Then $f(x_n)\to f(x^\ast)$. Moreover, if $r$ is a rate of divergence for $\sum_{i=0}^\infty\alpha_i=\infty$ and $L,K>0$ are such that $\sum_{i=0}^\infty \alpha_i^2\leq L$ and $\norm{x_0-x^\ast}^2\leq K$, then for all $\varepsilon>0$ and $g:\NN\to\NN$ we have
\begin{equation*}
\exists n\leq \Phi(\varepsilon,g)\, \forall k\in [n,n+g(n)]\, (f(x_k)\leq f(x^\ast)+\varepsilon)
\end{equation*} 
where
\begin{equation*}
\begin{aligned}
\Phi(\varepsilon,g)&:=\tilde{h}^{(\lceil 4\theta e/\varepsilon^2\rceil)}(0)\\
\tilde h(n)&:=r\left(n+g(n),\frac{\varepsilon}{2\theta}\right)+1\\[2mm]
e&:=\frac{\rho (L+K)}{2}+(\mu+2\rho)L\\
\theta&:=\rho+\mu
\end{aligned}
\end{equation*}
\end{theorem}

\begin{proof}
We show that conditions (i)--(vi) of Theorem \ref{thm:abstract:gradient} are satisfied, using details from the proof of Lemma 1 of \cite{alber-iusem-solodov:98:subgradient}, and the result then follows. Condition (\ref{cond:min}) holds because $x^\ast$ is a minimiser of $f$ on $C$ and $x_n\in C$ for all $n\in\NN$. For (\ref{cond:grad}) we use the definition of the subgradient: since $u_n\in \partial_{\varepsilon_n}f(x_n)$, for any $y\in C$ we have 
\begin{equation*}
f(x_n)-f(y)\leq \ip{u_n,x_n-y}+\varepsilon_n\leq \ip{u_n,x_n-y}+\mu \alpha_n
\end{equation*}
using $\varepsilon_n\leq \mu\alpha_n$, and so the property holds for $b_n:=\mu\alpha_n$ with $\sum_{i=0}^\infty \alpha_ib_i\leq \mu L$.  To establish condition (\ref{cond:descent:i}), setting $z_n:=x_n-(\alpha_n/\nu_n)u_n$ we have
\begin{equation}
\label{eqn:subgradient:-1}
\norm{x_{n+1}-x_n}=\norm{P_C(z_n)-P_C(x_n)}\leq \norm{z_n-x_n}=\frac{\alpha_n}{\nu_n}\norm{u_n}\leq \alpha_n
\end{equation}
using properties of the orthogonal projection (cf. Proposition 3 (i) of \cite{alber-iusem-solodov:98:subgradient}) and $\norm{u_n}\leq \nu_n$, and so property (\ref{cond:descent:i}) holds for $c_n:=\alpha_n$ and $\sum_{i=0}^\infty c_i^2\leq L$. For (\ref{cond:descent:ii}) we first observe that 
\begin{equation}
\label{eqn:subgradient:0}
\begin{aligned}
&\frac{\alpha_n}{\nu_n}\ip{u_n,x_n-x^\ast}=\ip{x_n-x^\ast,x_n-z_n}\\
&=\ip{x_n-x^\ast,x_n-x_{n+1}}+\ip{x_n-x^\ast,x_{n+1}-z_n}\\
&=\ip{x_n-x^\ast,x_n-x_{n+1}}+\ip{z_n-x_n,z_n-x_{n+1}}+\ip{x^\ast-z_n,z_n-x_{n+1}}
\end{aligned}
\end{equation}
By properties of the projection (cf. Proposition 3 (ii) of \cite{alber-iusem-solodov:98:subgradient}), we have $$\ip{z_n-x^\ast,z_n-P_C(z_n)}\geq 0$$ and using $P_C(z_n)=x_{n+1}$ we have
\begin{equation*}
\ip{x^\ast-z_n,z_n-x_{n+1}}=-\ip{z_n-x^\ast,z_n-P_C(z_n)}\leq 0
\end{equation*}
and therefore from (\ref{eqn:subgradient:0}), also using that $\nu_n\leq \max\{1,\rho\}\leq \rho$, we have
\begin{equation}
\label{eqn:subgradient:1}
\ip{\alpha_nu_n,x_n-x^\ast}\leq \rho(\ip{x_n-x^\ast,x_n-x_{n+1}}+\ip{z_n-x_n,z_n-x_{n+1}})
\end{equation}
Finally, we see that
\begin{equation*}
\begin{aligned}
\ip{z_n-x_n,z_n-x_{n+1}}&=\ip{z_n-x_n,z_n-x_n}+\ip{z_n-x_n,x_n-x_{n+1}}\\
&\leq \norm{z_n-x_n}^2+\norm{z_n-x_n}\norm{x_n-x_{n+1}}\\
&\leq 2\norm{z_n-x_n}^2 \ \ \ \mbox{by (\ref{eqn:subgradient:-1})}\\
&\leq \frac{2\alpha^2_n}{\nu^2_n}\norm{u_n}^2\\
&\leq 2\alpha^2_n \ \ \ \mbox{since $\norm{u_n}\leq \nu_n$}
\end{aligned}
\end{equation*}
Substituting this into (\ref{eqn:subgradient:1}), we have
\begin{equation*}
\ip{\alpha_nu_n,x_n-x^\ast}\leq \rho\ip{x_n-x^\ast,x_n-x_{n+1}}+2\rho\alpha_n^2
\end{equation*}
and so (\ref{cond:descent:ii}) is satisfied for $a=\rho$ and $d_n=2\rho\alpha_n^2$ with $\sum_{i=0}^\infty d_i\leq 2\rho L$. Finally, (\ref{cond:bound}) is holds by assumption since $\norm{u_n}\leq \rho$, and (\ref{cond:coeff}) holds since $\rho c_n+b_n=(\rho+\mu)\alpha_n$. Applying Theorem \ref{thm:abstract:gradient} with the relevant parameters i.e. $(a,b,c,d,p,\theta):=(\rho,\mu L,L,2\rho L,\rho,\rho+\mu)$, gives us the result.
\end{proof}

\begin{remark}
The proof of Theorem 1 of Alber, Iusem and Solodov \cite{alber-iusem-solodov:98:subgradient} uses Proposition \ref{prop:recineq1:beta:conv} for $\beta_n:=f(x_n)-f(x^\ast)$, establishing that $\sum_{i=0}^\infty \alpha_i\beta_i$ is bounded and therefore $\beta_n\to 0$. By Theorem \ref{thm:block}, there is in general no way of obtaining a computable rate of convergence from a bound on $\sum_{i=0}^\infty \alpha_i\beta_i$ (rather a computable rate of convergence on the series is needed). The difficulty of reading off rates of convergence from their proof of the whole sequence $\seq{x_n}$ rather than just a subsequence $\seq{x_{n_k}}$ is alluded to by the authors in \cite{alber-iusem-solodov:98:subgradient}, who write, in the final paragraph \emph{``This result does not give any information on the asymptotic behaviour of $\seq{f(x_k)}$ outside the subsequence $\seq{x^{l_k}}$''}. Theorem \ref{thm:block} gives a formal explanation as to why reading off a quantitative information on the whole sequence is difficult, although we are able to provide information on the behaviour of $\seq{f(x_k)}$ in the form of a rate of metastability. We stress that this does not mean that an \emph{alternative} proof technique would not yield a computable rate of convergence, at least in special cases.
\end{remark}

\begin{remark}
Note that in \cite{alber-iusem-solodov:98:subgradient}, technically the existence of some $\rho>0$ satisfying $\norm{u_n}\leq \rho$ for all $n\in\NN$ follows by establishing that the $\seq{x_n}$ are bounded, and then using an additional boundedness assumption for the subgradient, namely that $\partial_\varepsilon f$ is bounded on bounded sets. 
\end{remark}

%%%%%%%%%%%%%%%%%%%%%%%%%%%%%%%%%%%%%%%%%%%%%%%%%%%%%%%%%%%%%%%%%%%%%%%%%%%%%%%%%%%%%%%%%%
\subsection{Application of Case II: A survey of recent applications of Section \ref{sec:recineq:partii}}
\label{sec:applications:recent}
%%%%%%%%%%%%%%%%%%%%%%%%%%%%%%%%%%%%%%%%%%%%%%%%%%%%%%%%%%%%%%%%%%%%%%%%%%%%%%%%%%%%%%%%%%

Having given a new case study which uses, in a crucial way, results of Section \ref{sec:recineq:parti}, we conclude this section on applications of the results of Section \ref{sec:recineq} with a brief discussion of how quantitative results pertaining to the second condition $\gamma_n/\alpha_n\to 0$, discussed in Section \ref{sec:recineq:partii}, have been used recently in different contexts.

%%%%%%%%%%%%%%%%%%%%%%%%%%%%%%%%%%%%%%%%%%%%%%%%%%%%%%%%%%%%%%%%%%%%%%%%%%%%%%%%%%%%%%%%%%
\subsubsection{Mann schemes for asymptotically weakly contractive mappings}
\label{sec:applications:weakly:contractive}
%%%%%%%%%%%%%%%%%%%%%%%%%%%%%%%%%%%%%%%%%%%%%%%%%%%%%%%%%%%%%%%%%%%%%%%%%%%%%%%%%%%%%%%%%%

An illuminating example of how the main recursive inequality (\ref{eqn:recineq:basic}) with the second convergence condition $\gamma_n/\alpha_n\to 0$ can be applied is given by considering a slightly generalised notion of the class of $\psi$-weakly contractive mappings already discussed in Section \ref{sec:prelim:analysis}. Convergence becomes more interesting when one is confronted with asymptotic variants of these mappings, which become weakly contractive in the limit. A simple example is the following generalisation of the original scheme (\ref{eqn:wc}):
\begin{equation*}
\norm{T^nx-T^ny}\leq \norm{x-y}-\psi(\norm{x-y})+l_n
\end{equation*}
where $\seq{l_n}$ is some sequence of nonnegative reals converging to zero, and we recall that $\psi$ is a nondecreasing function with $\psi(0)=0$ which is positive and $(0,\infty)$. One can show in a very general setting that whenever $x^\ast$ is a fixpoint of $T$, the Krasnoselki-Mann scheme (\ref{eqn:km}) converges to $x^\ast$ for any step sizes $\seq{\alpha_n}$ with $\sum_{i=0}^\infty\alpha_i=\infty$. Indeed, for the simple example above, we see that
\begin{equation*}
\begin{aligned}
\norm{x_{n+1}-x^\ast}&=\norm{(1-\alpha_n)x_n+\alpha_nT^nx_n-x^\ast}\\
&=\norm{(1-\alpha_n)(x_n-x^\ast)+\alpha_n(T^nx_n-x^\ast)}\\
&\leq (1-\alpha_n)\norm{x_n-x^\ast}+\alpha_n\norm{T^nx_n-T^nx^\ast}\\
&\leq (1-\alpha_n)\norm{x_n-x^\ast}+\alpha_n(\norm{x_n-x^\ast}-\psi(\norm{x_n-x^\ast})+l_n)\\
&=\norm{x_n-x^\ast}-\alpha_n\psi(\norm{x_n-x^\ast})+\alpha_nl_n
\end{aligned}
\end{equation*}
which is an instance of (\ref{eqn:recineq:basic}) for $\mu_n:=\norm{x_n-x^\ast}$, $\beta_n:=\psi(\norm{x_n-x^\ast})$ and $\gamma_n:=\alpha_nl_n$. Moreover, we have $\gamma_n/\alpha_n=l_n\to 0$ and
\begin{equation*}
\beta_n=\psi(\mu_n)\leq \frac{\psi(\varepsilon)}{2}\implies \mu_n\leq \varepsilon
\end{equation*}
by monotonicity of $\psi$ (cf. Example \ref{ex:monotone}), and so Corollary \ref{cor:recineq2} can be applied to produce a rate of convergence for $x_n\to x^\ast$, assuming we have one for $l_n\to 0$. This corollary can in fact be utilised in a much more general setting to produce abstract quantitative convergence results for classes of weakly contractive mappings. The following is an example from a recent paper by the second author and Wiesnet \cite{powell-wiesnet:21:contractive}, where weakly contractive mappings are discussed in much more detail:
\begin{theorem}[cf. Theorem 4.1 of \cite{powell-wiesnet:21:contractive}]
\label{res:mann}
Let $X$ be a normed space with $E\subseteq X$, and suppose that $\seq{A_n}$ is a sequence of mappings $A_n:E\to X$, and $\psi:[0,\infty)\to [0,\infty)$ a nondecreasing function with $\psi(0)=0$ which is positive on $(0,\infty)$. Suppose that $\seq{k_n}$ is some sequence of nonnegative reals and $\sigma:(0,\infty)\times (0,\infty)\to \NN$ a modulus such that $\seq{A_n}$ satisfies
\begin{equation*}
\norm{x-x^\ast}\leq b\implies \norm{A_nx-x^\ast}\leq (1+k_n)\norm{x-x^\ast}-\psi(\norm{x-x^\ast})+\delta
\end{equation*}
for all $\delta,b>0$ and $n\geq\sigma(\delta,b)$. Suppose in addition that $\seq{x_n}$ is a sequence satisfying 
\begin{equation}
\label{eqn:mann}
x_{n+1}=(1-\alpha_n)x_n+\alpha_n A_nx_n
\end{equation}
where $\seq{\alpha_n}$ is some sequence in $(0,\alpha]$ such that $\sum_{n=0}^\infty \alpha_n=\infty$ with rate of divergence $r$ and $d>0$ is such that $\prod_{i=0}^n(1+\alpha_ik_i)\leq d$ for all $n\in\NN$. Then whenever there exists $c>0$ such that $\norm{x_n-x^\ast}\leq c$ for all $n\in\NN$, we have $\norm{x_n-x^\ast}\to 0$ with rate
\begin{equation*}
\Phi(\varepsilon):=r\left(\sigma\left(\frac{1}{2d}\min\left\{\psi\left(\frac{\varepsilon}{2d}\right),\frac{\varepsilon}{\alpha}\right\},c\right),2d\int_{\varepsilon/2d}^c \frac{dt}{\psi(t)} \right)+1
\end{equation*}
\end{theorem}
The rate of convergence stated here uses Corollary \ref{cor:recineq2}, though in a slightly refined form using monotonicity of $\psi$ to replace the second argument of $r$ with an integral. This instance of Corollary \ref{cor:recineq2} is given as Lemma 3.1 in \cite{powell-wiesnet:21:contractive}. Much more sophisticated results, also using Corollary \ref{cor:recineq2} in this way, are also given in \cite{powell-wiesnet:21:contractive}. 

%%%%%%%%%%%%%%%%%%%%%%%%%%%%%%%%%%%%%%%%%%%%%%%%%%%%%%%%%%%%%%%%%%%%%%%%%%%%%%%%%%%%%%%%%%
\subsubsection{Computing zeroes of set-valued accretive operators}
\label{sec:applications:acc}
%%%%%%%%%%%%%%%%%%%%%%%%%%%%%%%%%%%%%%%%%%%%%%%%%%%%%%%%%%%%%%%%%%%%%%%%%%%%%%%%%%%%%%%%%%

In our final illustration, we give an instance of how the modified conditions set out in Section \ref{sec:partii:second} can be used to produce rates of convergence for algorithms computing zeroes of set-valued accretive operators. Let $X$ be a Banach space and $X^\ast$ its dual space. The normalized duality mapping $J:X\to 2^{X^\ast}$ is defined by
\begin{equation*}
J(x):=\{j\in X^\ast \, | \, \ip{x,j}=\norm{x}^2=\norm{j}^2\}
\end{equation*}
An operator $A: D\to 2^X$ is said to be accretive if for all $x,y\in D$, if $u\in Ax$ and $v\in Ay$ then there exists some $j\in J(x-y)$ such that $\ip{u-v,j}\geq 0$. A notion of uniform accretivity was introduced by Garc\'ia-Falset in \cite{garcia-falset:05:accretive}, where $A$ is said to be uniformly accretive if there exists some nondecreasing $\phi:[0,\infty)\to [0,\infty)$ with $\phi(0)=0$ and $\phi $ positive on $(0,\infty)$ such that whenever $u\in Ax$ and $v\in Ay$ then there exists some $j\in J(x-y)$ such that
\begin{equation*}
\ip{u-v,j}\geq \phi(\norm{x-y})
\end{equation*}
For such operators, if there exists $x^\ast\in D$ such that $0\in Ax^\ast$, we can show that the following algorithm
\begin{equation*}
x_{n+1}=x_n-\alpha_n u_n \ \ \ \mbox{for} \ \ \ u_n\in Ax_{n+1}
\end{equation*}
for $\sum_{i=0}^\infty\alpha_i=\infty$ converges to $x^\ast$. Suppose that $K>0$ be such that $\norm{x_n-x^\ast}\leq K$ for all $n\in\NN$. Since $0\in Ax^\ast$ and $u_n\in Ax_{n+1}$, we know that there exists some $j\in J(x_{n+1}-x^\ast)$ be such that $\ip{u_n,j}\geq \phi(\norm{x_{n+1}-x^\ast})$. We then have
\begin{equation*}
\begin{aligned}
&\norm{x_{n+1}-x^\ast}^2=\ip{x_{n+1}-x^\ast,j} \ \ \ \mbox{by $j\in J(x_{n+1}-x^\ast)$}\\
&=\ip{x_n-x^\ast,j}-\alpha_n\ip{u_n,j} \ \ \ \mbox{by $x_{n+1}=x_n-\alpha_nu_n$}\\
&\leq \norm{x_n-x^\ast}\cdot \norm{x_{n+1}-x^\ast}-\alpha_n\ip{u_n,j} \ \ \ \mbox{since $\norm{j}=\norm{x_{n+1}-x}$}\\
&\leq \norm{x_n-x^\ast}\cdot \norm{x_{n+1}-x^\ast}-\alpha_n\phi(\norm{x_{n+1}-x^\ast}) \ \ \ \mbox{using $j\in J(x_{n+1}-x^\ast)$}\\
&\leq \norm{x_{n+1}-x^\ast}(\norm{x_n-x^\ast}-\alpha_n\phi(\norm{x_{n+1}-x^\ast})/K) \ \ \ \mbox{by $\norm{x_{n+1}-x^\ast}\leq K$}
\end{aligned}
\end{equation*}
and therefore
\begin{equation*}
\norm{x_{n+1}-x^\ast}\leq\norm{x_n-x^\ast}-\alpha_n\phi(\norm{x_{n+1}-x^\ast})/K
\end{equation*}
Defining $\mu_n:=\norm{x_n-x^\ast}$ and $\beta_n:=\phi(\norm{x_{n+1}-x^\ast})/K$, these then satisfy (\ref{eqn:recineq:basic}) for $\gamma_n=0$. Moreover, similarly to the previous section, we have
\begin{equation*}
\beta_n=\phi(\mu_{n+1})/K\leq \phi(\varepsilon)/2K \implies \mu_{n+1}\leq \varepsilon
\end{equation*}
and so we can apply Corollary \ref{cor:recineq2b} to obtain a rate of convergence. Much more sophisticated results are possible, which use the full power of Corollary \ref{cor:recineq2b} with $\gamma_n$ defined as some nontrivial error term. For example, we could replace the algorithm above with a perturbed scheme of the form
\begin{equation*}
x_{n+1}=x_n-\alpha_n u_n \ \ \ \mbox{for} \ \ \ u_n\in A_nx_{n+1}
\end{equation*}
where $\seq{A_n}$ is some sequence of accretive operators which approach $A$ in the Hausdorff metric. Similar calculations along the lines of the above give rise to an instance of (\ref{eqn:recineq:basic}) where $\gamma_n$ represents some error between $A_nx_{n+1}$ and $Ax_{n+1}$. A series of quantitative results along these lines are given by Kohlenbach and the second author \cite{kohlenbach-powell:20:accretive} and Koutsoukou-Argyraki \cite{koutsoukou:17:accretive}, using a so-called \emph{modulus of uniform accretivity at zero} defined by Kohlenbach and Koutsoukou-Argyraki \cite{kohlenbach-koutsoukou:15:cauchy} in place of the function $\phi$.

\begin{remark}
A variant of Corollary \ref{cor:recineq2b} was first used by Kohlenbach and K\"ornlein in \cite{kohlenbach-koernlein:11:pseudocontractive}, and more recently by Sipo\c{s} \cite{sipos:21:jointly:nonexpansive}. Both of these provide additional examples of how quantitative lemmas based on the second variant of Case II can be used in applied proof theory.
\end{remark}

%%%%%%%%%%%%%%%%%%%%%%%%%%%%%%%%%%%%%%%%%%%%%%%%%%%%%%%%%%%%%%%%%%%%%%%%%%%%%%%%%%%%%%%%%%
%%%%%%%%%%%%%%%%%%%%%%%%%%%%%%%%%%%%%%%%%%%%%%%%%%%%%%%%%%%%%%%%%%%%%%%%%%%%%%%%%%%%%%%%%%
\section{A note on implementation in the Lean theorem prover}
%%%%%%%%%%%%%%%%%%%%%%%%%%%%%%%%%%%%%%%%%%%%%%%%%%%%%%%%%%%%%%%%%%%%%%%%%%%%%%%%%%%%%%%%%%
%%%%%%%%%%%%%%%%%%%%%%%%%%%%%%%%%%%%%%%%%%%%%%%%%%%%%%%%%%%%%%%%%%%%%%%%%%%%%%%%%%%%%%%%%%
\label{sec:implementation}

As we have emphasised throughout, numerous convergence theorems in nonlinear analysis rely in a crucial way on the convergence properties of sequences of real numbers satisfying recursive inequalities. As mentioned at the beginning, a recent survey encompassing a multitude of such inequalities (including those studied here) has been recently published by Franci and Grammatico \cite{franci-grammatico:convergence:survey:22}. In line with this, numerous quantitative results in applied proof theory rely on a quantitative analysis of the relevant recursive inequality, a partial account of which is provided by this paper.

While the formal proof interpretations that underlie applied proof theory, such as the Dialectica interpretation, have been implemented in a number of theorem provers (most notably the Minlog system\footnote{\url{https://www.mathematik.uni-muenchen.de/~logik/minlog/}}), the majority of implementations tend to focus on the use of such translations as a verification strategy. In comparison, very little has been done on formalizing concrete applications of proof theory in mathematics, though the potential benefits of creating formal databases of ``proof-mined'' proofs have been suggested by Koutsoukou-Argyraki over the past few years, and are nicely outlined in her recent article \cite{koutsoukou:21:formalising}.

Significant progress on this front has been made in the last year or so by Cheval, who has developed several libraries in the Lean theorem prover\footnote{\url{https://github.com/hcheval}}, which implement not only the core logical techniques like the Dialectica interpretation, but also some logical metatheorems and quantitative results from his paper with Leu\c{s}tean \cite{cheval-leustean:21:tikhonovmann}, which includes the use of recursive inequalities to establish convergence.

In parallel, the authors of this paper also have started building a Lean library\footnote{\url{https://github.com/mneri123/Proof-mining-}}, the aim of which is to focus on recursive inequalities and the associated quantitative results. Formalizing these results is straightforward in the sense that they do not rely on extensive libraries of formal mathematics -- indeed, most of the ordinary convergence results require little more than the ability to reason about sequences and series of real numbers. At the same time, such a library is clearly useful for the formalization of both nonlinear analysis in general and the corresponding proof theoretic applications, as it takes care of a core component of the normally much more complex convergence proofs in concrete settings. Moreover, in addition to the obvious value of having a range of core lemmas implemented, one could conceivably develop automated methods which attempt to generate recursive inequalities from given assumptions, and then automatically infer convergence properties, along with numerical bounds. While the development of such automated reasoning algorithms is independent of formalisation, the existence of a formal library facilitates the testing of these algorithms, which could in turn be incorporated as tactics into the proof assistant itself.

Our own formalisation effort is currently in an early stage. Several convergence proofs for simple recursive inequalities have now been implemented, and some key constructions from computable analysis have been formalized (such as Proposition \ref{prop:specker:zero}), making use of Lean's computability libraries, and some well-known rates of metastability, such as that for monotone bounded sequences, have been verified. 

%%%%%%%%%%%%%%%%%%%%%%%%%%%%%%%%%%%%%%%%%%%%%%%%%%%%%%%%%%%%%%%%%%%%%%%%%%%%%%%%%%%%%%%%%%
%%%%%%%%%%%%%%%%%%%%%%%%%%%%%%%%%%%%%%%%%%%%%%%%%%%%%%%%%%%%%%%%%%%%%%%%%%%%%%%%%%%%%%%%%%
\section{Concluding remarks}
%%%%%%%%%%%%%%%%%%%%%%%%%%%%%%%%%%%%%%%%%%%%%%%%%%%%%%%%%%%%%%%%%%%%%%%%%%%%%%%%%%%%%%%%%%
%%%%%%%%%%%%%%%%%%%%%%%%%%%%%%%%%%%%%%%%%%%%%%%%%%%%%%%%%%%%%%%%%%%%%%%%%%%%%%%%%%%%%%%%%%
\label{sec:conc}

It is hoped that this paper has provided not only a thorough quantitative study of an important class of convergent sequences, but has also demonstrated the applicability of such results in analysing proofs in different areas of nonlinear analysis. We conclude by outlining some open questions.

An obvious direction for future work is a more in-depth proof theoretic study of subgradient type methods. There are numerous extensions of Alber, Iusem and Solodov \cite{alber-iusem-solodov:98:subgradient}, many of which use variants of Proposition \ref{prop:recineq1:beta:conv}, and so the results of Section \ref{sec:recineq:parti} could potentially be used to produce rates of convergence or metastability in different settings. For example, Alber and Iusem \cite{alber-iusem:01:subgradient} establish convergence results for subgradient methods in uniformly convex and uniformly smooth Banach spaces, and we anticipate that our quantitative lemmas could be applied to extract quantitative information which would depend on the corresponding moduli of convexity and uniform smoothness.

A much broader challenge would be to explore further types of recursive inequalities, which have been nicely organised and classified by Franci and Grammatico \cite{franci-grammatico:convergence:survey:22}. Particularly interesting here are stochastic recursive inequalities, which establish almost sure convergence of real-valued random variables, and in addition to standard methods used to prove convergence of sequences of real numbers, make crucial use of Martingale convergence theorems. Here we anticipate we could use recent work on formulating metastable variants of almost everywhere convergence by Avigad, Dean and Rute \cite{avigad-dean-rute:12:dominated}, along with numerical results on upcrossings as used in Avigad, Gerhardy and Towsner \cite{avigad-gerhardy-towsner:10:local}.

Finally, as already alluded to above, work on formalizing convergence results for recursive inequalities in theorem provers raises a natural question of whether one can devise algorithms for reducing concrete convergence statements to a particular recursive inequality, thereby generating a convergence proof automatically. In practice, the reduction often involves little more than the systematic application of basic algebraic operations along with standard facts about norms or inner products, and it is conceivable that for a given method $\seq{x_n}$ converging towards some point $x^\ast$, a search algorithm could potentially generate recursive inequalities satisfied by $\norm{x_n-x^\ast}$, thereby providing an automated method for proving the convergence of numerical algorithms, along with quantitative information in the form of a rate of convergence or metastability. The latter would form an instance of the general program of formalised applied proof theory as outlined by Koutsoukou-Argyraki \cite{koutsoukou:21:formalising}.\medskip

\noindent\textbf{Acknowledgements.} The authors are grateful to the anonymous referee for their extremely detailed report, which improved the presentation of the paper considerably. The first author was partially supported by the EPSRC Centre for Doctoral Training in Digital Entertainment (EP/L016540/1).

%%%%%%%%%%%%%%%%%%%%   End of main body of article
%
%                             References
%
%   BiBTeX users uncomment the following line:
%
\bibliographystyle{acm}
\bibliography{../../tpbiblio}

\begin{thebibliography}{10}

\bibitem{alber-guerredelabriere:97:weaklycontractive}
{\sc Alber, Y., and Guerre-Delabriere, S.}
\newblock Principle of weakly contractive maps in {H}ilbert spaces.
\newblock In {\em New Results in Operator Theory and its Applications\/}
  (1997), I.~Gohberg and Y.~Lyubich, Eds., vol.~98, pp.~7--22.

\bibitem{alber-guerredelabriere:01:projection}
{\sc Alber, Y., and Guerre-Delabriere, S.}
\newblock On the projection methods for fixed point problems.
\newblock {\em Analysis: International mathematical journal of analysis and its
  applications 21\/} (2001), 17--39.

\bibitem{alber-iusem:01:subgradient}
{\sc Alber, Y., and Iusem, A.~N.}
\newblock Extension of subgradient techniques for nonsmooth optimization in
  {B}anach spaces.
\newblock {\em Set-Valued Analysis 9\/} (2001), 315--335.

\bibitem{alber-iusem-solodov:98:subgradient}
{\sc Alber, Y., Iusem, A.~N., and Solodov, M.~V.}
\newblock On the projected subgradient method for nonsmooth convex optimization
  in a {H}ilbert space.
\newblock {\em Mathematical Programming 81}, 1 (1998), 23--35.

\bibitem{arthan-oliva:21:borel-cantelli}
{\sc Arthan, R., and Oliva, P.}
\newblock On the {B}orel-{C}antelli {L}emma, the {E}rd\'os-{R}\'enyi {T}heorem,
  and the {K}ochen-{S}tone {T}heorem.
\newblock {\em Journal of Logic and Analysis 13}, 6 (2021), 1--23.

\bibitem{avigad-dean-rute:12:dominated}
{\sc Avigad, J., Dean, E.~T., and Rute, J.}
\newblock A metastable dominated convergence theorem.
\newblock {\em Journal of Logic \& Analysis 4}, 3 (2012), 1--19.

\bibitem{avigad-gerhardy-towsner:10:local}
{\sc Avigad, J., Gerhardy, P., and Towsner, H.}
\newblock Local stability of ergodic averages.
\newblock {\em Transactions of the American Mathematical Society 362}, 1
  (2010), 261--288.

\bibitem{cheval-kohlenbach-leustean:23:halpern}
{\sc Cheval, H., Kohlenbach, U., and Leu\c{s}tean, L.}
\newblock On modified {H}alpern and {T}ikhonov-{M}ann iterations.
\newblock {\em Journal of Optimization Theory and Applications 197\/} (2023),
  233--251.

\bibitem{cheval-leustean:21:tikhonovmann}
{\sc Cheval, H., and Leu\c{s}tean, L.}
\newblock Quadratic rates of asymptotic regularity for the {T}ikhonov-{M}ann
  iteration.
\newblock {\em Optimization Methods and Software\/} (2022).
\newblock Electronic publication ahead of print.

\bibitem{dinis-pinto:20:proximal}
{\sc Dinis, B., and Pinto, P.}
\newblock Metastability of the proximal point algorithm with multi-parameters.
\newblock {\em Portugaliae Mathematica 77}, 3--4 (2020), 345--381.

\bibitem{dinis-pinto:21:effective}
{\sc Dinis, B., and Pinto, P.}
\newblock Effective metastability for a method of alternating resolvents.
\newblock To appear in: Fixed Point Theory, 2021.

\bibitem{dinis-pinto:21:proximal}
{\sc Dinis, B., and Pinto, P.}
\newblock Quantitative results on the multi-parameters proximal point
  algorithm.
\newblock {\em Journal of Convex Analysis 28}, 3 (2021), 729--750.

\bibitem{dinis-pinto:21:strong}
{\sc Dinis, B., and Pinto, P.}
\newblock Strong convergence for the alternating halpern-mann iteration in
  cat$(0)$ spaces.
\newblock To appear in: SIAM Journal on Optimization, 2021.

\bibitem{franci-grammatico:convergence:survey:22}
{\sc Franci, B., and Grammatico, S.}
\newblock Convergence of sequences: A survey.
\newblock {\em Annual Reviews in Control 53\/} (2022), 161--186.

\bibitem{freund-kohlenbach:23:ergodic}
{\sc Freund, A., and Kohlenbach, U.}
\newblock Bounds for a nonlinear ergodic theorem for banach spaces.
\newblock {\em Ergodic Theory and Dynamical Systems 43\/} (2023), 1570--1593.

\bibitem{garcia-falset:05:accretive}
{\sc Garc\'ia-Falset, J.}
\newblock The asymptotic behavior of the solutions of the {C}auchy problem
  generated by $\phi$-accretive operators.
\newblock {\em Journal of Mathematical Analysis and Applications 310\/} (2005),
  594--608.

\bibitem{drummond-iusem:04:projected:gradient}
{\sc Gra\~{n}a Drummond, L.~M., and Iusem, A.~N.}
\newblock A projected gradient method for vector optimization problems.
\newblock {\em Optimization Problems Computational Optimization and
  Applications 28\/} (2004), 5--29.

\bibitem{iiduke-yamada:09:subgradient}
{\sc Iiduka, H., and Yamada, I.}
\newblock A subgradient-type method for the equilibrium problem over the fixed
  point set and its applications.
\newblock {\em Optimization 58}, 2 (2009), 251--261.

\bibitem{kohlenbach:01:borweinreichshafrir}
{\sc Kohlenbach, U.}
\newblock A quantitative version of a theorem due to
  {B}orwein-{R}eich-{S}hafrir.
\newblock {\em Numerical Functional Analysis and Optimization 22\/} (2001),
  641--656.

\bibitem{kohlenbach:05:metatheorems}
{\sc Kohlenbach, U.}
\newblock Some logical metatheorems with applications in functional analysis.
\newblock {\em Transactions of the American Mathematical Society 357}, 1
  (2005), 89--128.

\bibitem{kohlenbach:08:book}
{\sc Kohlenbach, U.}
\newblock {\em {Applied Proof Theory: Proof Interpretations and their Use in
  Mathematics}}.
\newblock Springer Monographs in Mathematics. Springer, 2008.

\bibitem{kohlenbach:17:recent}
{\sc Kohlenbach, U.}
\newblock Recent progress in proof mining in nonlinear analysis.
\newblock {\em IFCoLoG Journal of Logics and their Applications 10}, 4 (2017),
  3361--3410.

\bibitem{kohlenbach:19:nonlinear:icm}
{\sc Kohlenbach, U.}
\newblock Proof-theoretic methods in nonlinear analysis.
\newblock In {\em Proceedings of the International Congress of Mathematicians
  2018}, vol.~2. World Scientific, 2019, pp.~61--82.

\bibitem{kohlenbach:20:halpern}
{\sc Kohlenbach, U.}
\newblock {Quantitative analysis of a Halpern-type Proximal Point Algorithm for
  accretive operators in Banach spaces}.
\newblock {\em Journal of Nonlinear and Convex Analysis 21}, 9 (2020),
  2125--2138.

\bibitem{kohlenbach:pp:kreisel}
{\sc Kohlenbach, U.}
\newblock {K}reisel's `shift of emphasis' and contemporary proof mining.
\newblock To appear in: `Studies in Logic' (China), 2022.

\bibitem{kohlenbach-koernlein:11:pseudocontractive}
{\sc Kohlenbach, U., and K\"ornlein, D.}
\newblock Effective rates of convergence for {L}ipschitzian pseudocontractive
  mappings in general {B}anach spaces.
\newblock {\em Nonlinear Analysis 74\/} (2011), 5253--5267.

\bibitem{kohlenbach-koutsoukou:15:cauchy}
{\sc Kohlenbach, U., and Koutsoukou-Argyraki, A.}
\newblock Rates of convergence and metastability for abstract cauchy problems
  generated by accretive operators.
\newblock {\em Journal of Mathematical Analysis and Applications 423\/} (2015),
  1089--1112.

\bibitem{kohlenbach-lambov:04:asymptotically:nonexpansive}
{\sc Kohlenbach, U., and Lambov, B.}
\newblock Bounds on iterations of asymptotically quasi-nonexpansive mappings.
\newblock In {\em Proceedings of the International Conference on Fixed Point
  Theory and Applications\/} (2004), Yokohama Publishers, pp.~143--172.

\bibitem{kohlenbach-leustean:10:asymp:nonexpansive:hyperbolic}
{\sc Kohlenbach, U., and Leu\c{s}tean, L.}
\newblock Asymptotically nonexpansive mappings in uniformly convex hyperbolic
  spaces.
\newblock {\em Journal of the European Mathematical Society 12}, 1 (2010),
  71--92.

\bibitem{kohlenbach-leustean:12:metastability}
{\sc Kohlenbach, U., and Leu\c{s}tean, L.}
\newblock Effective metastability of {H}alpern iterates in {CAT(0)} spaces.
\newblock {\em Advances in Mathematics 321\/} (2012), 2526--2556.

\bibitem{kohlenbach-lopezacedo-nicolae:21:lionman}
{\sc Kohlenbach, U., L\'opez-Acedo, G., and Nicolae, A.}
\newblock A uniform betweenness property in metric spaces and its role in the
  quantitative analysis of the ``{L}ion-{M}an'' game.
\newblock {\em Pacific Journal of Mathematics 310}, 1 (2021), 181--212.

\bibitem{kohlenbach-pinto:21:viscosity}
{\sc Kohlenbach, U., and Pinto, P.}
\newblock Quantitative translations for viscosity approximation methods in
  hyperbolic spaces.
\newblock Preprint, 2021.

\bibitem{kohlenbach-powell:20:accretive}
{\sc Kohlenbach, U., and Powell, T.}
\newblock Rates of convergence for iterative solutions of equations involving
  set-valued accretive operators.
\newblock {\em Computers and Mathematics with Applications 80\/} (2020),
  490--503.

\bibitem{koernlein:16:phd}
{\sc K\"ornlein, D.}
\newblock {\em Quantitative Analysis of Iterative Algorithms in Fixed Point
  Theory and Convex Optimization}.
\newblock PhD thesis, Technische Universit\"at Darmstadt, 2016.

\bibitem{koernlein:pp:gradient}
{\sc K\"ornlein, D.}
\newblock Quantitative strong convergence for the hybrid steepest descent
  method.
\newblock Preprint, available at \url{https://arxiv.org/abs/1610.00517}, 2016.

\bibitem{koutsoukou:17:accretive}
{\sc Koutsoukou-Argyraki, A.}
\newblock Effective rates of convergence for the resolvents of accretive
  operators.
\newblock {\em Numerical Functional Analysis and Optimization 38}, 12 (2017),
  1601--1613.

\bibitem{koutsoukou:21:formalising}
{\sc Koutsoukou-Argyraki, A.}
\newblock On preserving the computational content of mathematical proofs: Toy
  examples for a formalising strategy.
\newblock In {\em Connecting with Computability. CiE 2021\/} (2021), vol.~12813
  of {\em LNCS}, Springer, pp.~285--296.

\bibitem{kreisel:51:proofinterpretation:part1}
{\sc Kreisel, G.}
\newblock On the interpretation of non-finitist proofs, {P}art {I}.
\newblock {\em Journal of Symbolic Logic 16\/} (1951), 241--267.

\bibitem{kreisel:52:proofinterpretation:part2}
{\sc Kreisel, G.}
\newblock On the interpretation of non-finitist proofs, {P}art {II}:
  Interpretation of number theory.
\newblock {\em Journal of Symbolic Logic 17\/} (1952), 43--58.

\bibitem{leaustean:07:halpern}
{\sc Leu\c{s}tean, L.}
\newblock Rates of asymptotic regularity for halpern iterations of nonexpansive
  mappings.
\newblock {\em Journal of Universal Computer Science 13}, 11 (2007),
  1680--1691.

\bibitem{leaustean-nicolae:14:compositions}
{\sc Leu\c{s}tean, L., and Nicolae, A.}
\newblock Effective results on compositions of nonexpansive mappings.
\newblock {\em Journal of Mathematical Analysis and Applications 410\/} (2014),
  902--907.

\bibitem{leaustean-nicolae:16:kappa}
{\sc Leu\c{s}tean, L., and Nicolae, A.}
\newblock Effective results on nonlinear ergodic averages in cat$(\kappa)$
  spaces.
\newblock {\em Ergodic Theory and Dynamical Systems 36}, 18 (2016), 2580--2601.

\bibitem{leaustean-pinto:21:halpern}
{\sc Leu\c{s}tean, L., and Pinto, P.}
\newblock Quantitative results on a {H}alpern-type proximal point algorithm.
\newblock {\em Computational Optimization and Applications 79}, 1 (2021),
  101--125.

\bibitem{mainge:08:projected:subgradient}
{\sc Maing\'e, P.-E.}
\newblock Strong convergence of projected subgradient methods for nonsmooth and
  nonstrictly convex minimization.
\newblock {\em Set-valued analysis 16\/} (2008), 899–--912.

\bibitem{pinto:21:halpern}
{\sc Pinto, P.}
\newblock A rate of metastability for the {H}alpern type proximal point
  algorithm.
\newblock {\em Numerical Functional Analysis and Optimization 42}, 3 (2021),
  320--343.

\bibitem{pischke:pp:metatheorem}
{\sc Pischke, N.}
\newblock Logical metatheorems for accretive and (generalized) monotone
  set-valued operators.
\newblock To appear in: Journal of Mathematical Logic, 2023.

\bibitem{pischke-kohlenbach:21:subgradient}
{\sc Pischke, N., and Kohlenbach, U.}
\newblock Quantitative analysis of a subgradient-type method for equilibrium
  problems.
\newblock {\em Numerical Algorithms 90\/} (2022), 197--219.

\bibitem{powell:23:littlewood}
{\sc Powell, T.}
\newblock A finitization of {L}ittlewood's {T}auberian theorem and an
  application in {T}auberian remainder theory.
\newblock {\em Annals of Pure and Applied Logic 174}, 4 (2023), 103231.

\bibitem{powell-wiesnet:21:contractive}
{\sc Powell, T., and Wiesnet, F.}
\newblock Rates of convergence for asymptotically weakly contractive mappings
  in normed spaces.
\newblock {\em Numerical Functional Analysis and Optimization 42}, 15 (2021),
  1802--1838.

\bibitem{qihou:01:recineq}
{\sc Qihou, L.}
\newblock Iteration sequences for asymptotically quasi-nonexpansive mappings
  with error member.
\newblock {\em Journal of Mathematical Analysis and Applications 259\/} (2001),
  18--24.

\bibitem{simmons:towsner:19:polyrings}
{\sc Simmons, W., and Towsner, H.}
\newblock Proof mining and effective bounds in differential polynomial rings.
\newblock {\em Advances in Mathematics 343\/} (2019), 567--623.

\bibitem{sipos:21:jointly:nonexpansive}
{\sc Sipo\c{s}, A.}
\newblock Revisiting jointly firmly nonexpansive families of mappings.
\newblock {\em Optimization\/} (2021).
\newblock Electronic publication ahead of print.

\bibitem{sipos:22:abstract}
{\sc Sipo\c{s}, A.}
\newblock Abstract strongly convergent variants of the proximal point
  algorithm.
\newblock {\em Computational Optimization and Applications 83}, 1 (2022),
  349--380.

\bibitem{specker:49:sequence}
{\sc Specker, E.}
\newblock {N}icht konstruktiv beweisbare {S}{\"a}tze der {A}nalysis.
\newblock {\em Journal of Symbolic Logic 14\/} (1949), 145--158.

\bibitem{tao:07:softanalysis}
{\sc Tao, T.}
\newblock Soft analysis, hard analysis, and the finite convergence principle.
\newblock Essay posted 23 May 2007, 2007.
\newblock Appeared in: ‘T. Tao, Structure and Randomness: Pages from Year One
  of a Mathematical Blog. AMS, 298pp., 2008’.

\bibitem{tao:08:ergodic}
{\sc Tao, T.}
\newblock Norm convergence of multiple ergodic averages for commuting
  transformations.
\newblock {\em Ergodic Theory and Dynamical Systems 28\/} (2008), 657--688.

\bibitem{yamada:01:hybrid}
{\sc Yamada, I.}
\newblock The hybrid steepest descent method for the variational inequality
  problem over the intersection of fixed point sets of nonexpansive mappings.
\newblock In {\em Inherently Parallel Algorithms in Feasibility and
  Optimization and their Applications}, vol.~8 of {\em Studies in Computational
  Mathematics}. Elsevier, 2001, pp.~473--504.

\end{thebibliography}

%\begin{thebibliography}
%Please ensure that the references of all works with a doi include the doi.
%\end{thebibliography}

\end{document}